\documentclass[12pt,final]{iopart}
\input{def}

\begin{document}

%
%
\def\addressncsu{Department of Mathematics, North Carolina State University, Raleigh, NC, USA} 
\def\addressauk{Department of Engineering Science, University of Auckland, New Zealand}
\def\addressmerced{Department of Applied Mathematics, University of California, Merced, CA, USA}

\title[OED under uncertainty]{Optimal design of large-scale nonlinear Bayesian 
inverse problems under model uncertainty}


\author{Alen Alexanderian$^1$, Ruanui Nicholson$^2$, and Noemi Petra$^3$}
\address{
$^1$\addressncsu\\
$^2$\addressauk\\
$^3$\addressmerced\\
}
\ead{\url{alexanderian@ncsu.edu},
     \url{ruanui.nicholson@auckland.ac.nz},
     \url{npetra@ucmerced.edu}}
     
\begin{abstract}
  We consider optimal experimental design (OED) for Bayesian nonlinear inverse
  problems governed by partial differential equations (PDEs) under model
  uncertainty.  Specifically, we consider inverse problems in which, in addition
  to the inversion parameters, the governing PDEs include secondary uncertain
  parameters. We focus on problems with infinite-dimensional inversion and
  secondary parameters and present a scalable computational framework for optimal 
  design of such problems.  The proposed approach enables Bayesian inversion and
  OED under uncertainty within a unified framework.  We build on the Bayesian
  approximation error (BAE) approach, to incorporate modeling uncertainties in
  the Bayesian inverse problem, and methods for A-optimal design of
  infinite-dimensional Bayesian nonlinear inverse problems. Specifically, a
  Gaussian approximation to the posterior at the maximum a posteriori probability
  point is used to define an uncertainty aware OED objective that is tractable to
  evaluate and optimize. In particular, the OED objective can be computed at a
  cost, in the number of PDE solves, that does not grow with the dimension of the
  discretized inversion and secondary parameters. The OED problem is formulated
  as a binary bilevel PDE constrained optimization problem and a greedy
  algorithm, which provides a pragmatic approach, is used to find optimal
  designs. We demonstrate the effectiveness of the proposed approach for a model
  inverse problem governed by an elliptic PDE on a three-dimensional
  domain. Our computational results also highlight the pitfalls of
  ignoring modeling uncertainties in the OED and/or inference stages.
\end{abstract}



\section{Introduction}
Models governed by partial differential equations (PDEs) are common in science
and engineering applications.  Such PDE models often contain parameters that
need to be estimated using observed data and the model. This requires solving
an inverse problem.  The quality of the estimated parameters is influenced
significantly by the quantity and quality of the measurement data.  Therefore,
optimizing the data acquisition process is crucial.  This requires solving an
optimal experimental design (OED) problem~\cite{AtkinsonDonev92,
ChalonerVerdinelli95,Ucinski05}.  In the present work, we focus on
inverse problems in which measurement data are collected at a set of sensors.
In this case, OED amounts to finding an optimal sensor placement.  In this
context, OED is especially important when only a few sensors can be deployed. 

While some parameters in the governing PDEs can be estimated by solving an
inverse problem, often there are additional uncertain model parameters that are
not being estimated. These parameters might be too
costly or impossible to estimate.  We call the uncertain parameters that are
being estimated in an inverse problem the \emph{inversion parameters} and refer
to the additional uncertain parameters as the \emph{secondary parameters}.
Such secondary parameters have also been referred to as auxiliary parameters,
latent parameters, or nuisance parameters in the literature. When solving
inverse problems with secondary parameters with significant uncertainty levels,
both the parameter estimation and data acquisition processes need to be aware
of such uncertainties.  In this article, we present a computational framework
for optimal design of nonlinear Bayesian inverse problems governed by PDEs
under model uncertainty.

Uncertainties in mathematical models can be divided into two classes: reducible
and irreducible~\cite{Smith13}. The reducible uncertainties are epistemic
uncertainties that can be reduced via statistical parameter estimation.  On the
other hand, irreducible uncertainties are either aleatoric uncertainties
inherent to the model that are impossible to reduce or are epistemic
uncertainties that are too costly or impractical to reduce.  Also, in some
applications we might have access to a probabilistic description of secondary
model parameters from previous studies and further reduction of the uncertainty
in such parameters may not be worth the additional computational cost.  We
consider such uncertainties as irreducible as well.  In the present work, we
focus on the case of irreducible uncertainties.


We consider models of the form
\begin{equation}\label{equ:model}
\obs = \afwd(\iparm,\iparb) + \vec\eta,
\end{equation}
where $\obs$ is a vector of measured data, $\afwd$ is a PDE-based model,
$\iparm$ and $\iparb$ are uncertain parameters, and $\vec\eta$ is a random
vector that models measurement noise. Herein, $\iparm$ is the inversion
parameter which we seek to infer, and $\iparb$ is a secondary uncertain model
parameter.  We assume the uncertainty in $\iparb$ to be irreducible. The
parameters $\iparm$ and $\iparb$ are assumed to be independent random variables
that take values in infinite-dimensional real separable Hilbert spaces $\hilbm$
and $\hilbb$, respectively.  Moreover, $\afwd$ is assumed to be nonlinear in
both $\iparm$ and $\iparb$.  The methods presented in this article enable
computing optimal experimental designs in such a way that the uncertainty in
the secondary parameters is accounted for.

\paragraph{Related work}
In the recent years, there have been numerous research efforts directed at
OED in inverse problems governed by PDEs.  See the
review~\cite{Alexanderian20}, for a survey of the recent literature in this
area.  There has also been an increased interest in parameter inversion and
design of experiments in systems governed by uncertain forward models; see
e.g.,~\cite{KolehmainenTarvainenArridgeEtAl11,
Aravkin12VanLeeuwe12,KaipioKolehmainen13,NicholsonPetraKaipio18,
ConstantinescuBessacPetraEtAl20,SimpsonBakerBuenzliEtAl22} 
for a small sample of the literature addressing
inverse problems under uncertainty.  Methods for OED in such inverse problems
have been studied
in~\cite{KovalAlexanderianStadler20,AlexanderianPetraStadlerEtAl21,
FengMarzouk19,BartuskaEspathTempone22}. The
works~\cite{KovalAlexanderianStadler20,AlexanderianPetraStadlerEtAl21} concern
optimal design of infinite-dimensional Bayesian linear inverse problems
governed by PDEs.  Specifically,~\cite{KovalAlexanderianStadler20} considers
design of linear inverse problems governed by PDEs with irreducible sources of
model uncertainty. On the other hand,~\cite{AlexanderianPetraStadlerEtAl21}
targets OED for linear inverse problems with reducible sources of uncertainty.
The efforts~\cite{FengMarzouk19,BartuskaEspathTempone22} focus on inverse
problems with finite- and low-dimensional inversion and
secondary parameters.  These articles devise sampling
approaches for estimating the expected information gain in such problems.  The
setting considered in~\cite{FengMarzouk19} is that of inverse problems with
reducible uncertainties.  The approach in~\cite{BartuskaEspathTempone22}
employs a small noise approximation that applies to problems with nuisance
parameters that have small uncertainty levels.


\paragraph{Our approach}
We focus on Bayesian nonlinear inverse problems governed by PDEs with
infinite-dimensional inversion and secondary parameters.  Traditionally, when
solving such inverse problems all the secondary model parameters are fixed at
their nominal values and the focus is on the estimation of the inversion
parameters.  Considering~\cref{equ:model}, this amounts to using the
\emph{approximate model} $\cfwd(\iparm) \defeq \afwd(\iparm, \bar\iparb)$, where
$\bar\iparb$ is some nominal value.  
In~\Cref{sec:motivation}, we consider
simpler instances of~\cref{equ:model} and illustrate the role of model
uncertainty in Bayesian inverse problems and the importance of accounting for
such uncertainties in parameter inversion and OED. 
The discussion in~\Cref{sec:motivation} 
motivates our approach for incorporating secondary
uncertainties in nonlinear Bayesian inverse problems and the corresponding OED
problems. This is done using the Bayesian approximation error
(BAE) approach~\cite{KaipioSomersalo07,KaipioKolehmainen13};
see~\Cref{sec:BAE}.

Subsequently, we build on
methods for optimal design of infinite-dimensional nonlinear inverse
problems~\cite{AlexanderianPetraStadlerEtAl16,WuChenGhattas23,Alexanderian20}
to derive an uncertainty aware OED objective. Specifically, we follow an
A-optimal design strategy where the goal is to obtain designs that minimize
average posterior variance. To cope with the non-Gaussianity of the posterior,
we rely on a Gaussian approximation to the posterior. This enables deriving 
approximate measures of posterior uncertainty that are tractable to optimize for
infinite-dimensional inverse problems; see \Cref{sec:OED_BAE}.  We then present
two approaches for formulating and computing the OED objective;
see~\Cref{sec:methods}.  The first approach uses Monte Carlo trace estimators
and the second one formulates the OED problem as an eigenvalue optimization
problem. In each case, the OED problem is formulated as a bilevel binary
PDE-constrained optimization problem. In both approaches, 
the cost of computing the OED objective, in terms of the number of PDE solves, 
is independent of the dimensions of discretized
inversion and secondary parameters.  This makes these approaches suitable for
large-scale applications.  In the present work, we rely on a greedy approach to
solve the resulting optimization problems.   As discussed
in~\Cref{sec:methods}, a greedy algorithm is especially suited for the
formulation of the OED problem in~\Cref{sec:lowrank}, as a binary
PDE-constrained eigenvalue optimization problem. 

We elaborate the proposed approach in the context of a model inverse problem
governed by an elliptic PDE in a three-dimensional domain;
see~\Cref{sec:model}.  This inverse problem, which is motivated by heat
transfer applications, concerns estimation of a coefficient field on the bottom
boundary of the domain, using sensor measurements of the temperature 
on the top boundary. The secondary uncertain parameter in this inverse problem
is the log-conductivity field, which is modeled as a random field on the
three-dimensional physical domain. Our computational results
in~\Cref{sec:numerics} demonstrate the effectiveness of the proposed strategy
in computing optimal sensor placements under model uncertainty. We also
systematically study the drawbacks of ignoring the uncertainty in the Bayesian
inversion and experimental design stages. These studies illustrate the fact
that ignoring uncertainty in OED or inference stages can lead to inferior
designs and highly inaccurate results in parameter estimation.

\paragraph{Contributions}
The contributions of this article are as follows: 
%
%
%
(1) We present an uncertainty aware formulation of the OED problem, uncertainty
aware OED objectives along with scalable methods for computing them, and an
extensible optimization framework for computing optimal designs.  These make
OED for nonlinear Bayesian inverse problems governed by PDEs with
infinite-dimensional inversion and secondary parameters feasible.
Additionally, the proposed approach enables Bayesian inversion and OED under
uncertainty within a unified framework. 
(2) We elaborate the proposed approach for an inverse problem governed by an
elliptic PDE, on a three-dimensional domain, with infinite-dimensional inversion
and secondary parameters. This is used to elucidate the implementation of our
proposed approach for design of inverse problems governed by PDEs under
uncertainty.  
(3) We present comprehensive computational studies that illustrate the
effectiveness of the proposed approach and also the importance of accounting
for modeling uncertainties in both parameter inversion and experimental design
stages.  (4) By considering simpler instances of~\cref{equ:model}, in
\Cref{sec:motivation}, we present a systematic study of the role of model
uncertainty in Bayesian inverse problems and the importance of accounting for
such uncertainties in parameter inversion and OED. That study also reveals a
connection between the BAE-based approach taken in the present work and the
method for OED in linear inverse problems under uncertainty
in~\cite{AlexanderianPetraStadlerEtAl21}.

The developments in this article point naturally to a number of extensions of
the presented methods.  We discuss such issues in~\Cref{sec:conclusion}, where
we present our concluding remarks, and discuss potential limitations of the
presented approach and opportunities for future extensions.

\section{Motivation and overview}\label{sec:motivation}
In this section, we motivate our approach for OED under uncertainty and set the
stage for the developments in the rest of the article.  After a brief coverage
of requisite notation and preliminaries in~\Cref{sec:prelim}, we begin our
discussion in~\Cref{sec:linear_additive} by considering a simple form
of~\cref{equ:model} where $\afwd$ is linear in $\iparm$ and $\iparb$.
This facilitates an intuitive study of the role of model uncertainty in
Bayesian inverse problems and OED.  We then consider nonlinear models of
varying complexity in~\Cref{sec:nonlinear} to motivate our approach for design
of inverse problems governed by nonlinear models of the form~\cref{equ:model}. 

\subsection{Preliminaries}\label{sec:prelim}
In this article, we consider inversion and secondary parameters that take
values in infinite-dimensional Hilbert spaces.  For a Hilbert
space $\hilb$, we denote the corresponding inner product by
$\ipp{\hilb}{\cdot}{\cdot}$ and the associated induced norm by $\| \cdot
\|_\hilb$; i.e., $\| \cdot \|_\hilb \defeq \ipp{\hilb}{\cdot}{\cdot}^{1/2}$.
For Hilbert spaces $\hilb_1$ and $\hilb_2$, 
$\L(\hilb_1, \hilb_2)$ denotes the space of bounded linear transformations from
$\hilb_1$ to $\hilb_2$.  The space of bounded linear operators on a Hilbert
space $\hilb$ is denoted by $\L(\hilb)$, and the subspace of bounded selfadjoint
operators is denoted by $\Lsym(\hilb)$. We let $\Lsymp(\hilb)$ denote the set of bounded
positive selfadjoint operators. The subspace 
of trace-class operators in $\L(\hilb)$ is denoted by $\Lt(\hilb)$, and
the subspace of selfadjoint trace-class operators is denoted by
$\Ltsym(\hilb)$. Also, the sets of positive and strictly positive
selfadjoint trace-class operators are denoted by $\Ltsymp(\hilb)$ and
$\Ltsympp(\hilb)$, respectively.

Throughout the article, $\GM{a}{\C}$ denotes a Gaussian measure with
mean $a$ and covariance operator $\C$.  For a Gaussian measure on 
an infinite-dimensional Hilbert space $\hilb$,
the covariance operator
$\C$ is required to be in $\Ltsymp(\hilb)$.  Herein, we consider
non-degenerate Gaussian measures; i.e., we assume that $\C \in
\Ltsympp(\hilb)$. For further details on Gaussian measures, we refer
to~\cite{DaPrato06,Stuart10}. Also, when considering measures on Hilbert
spaces, we equip these spaces with their associated Borel sigma algebra.
Throughout the article, for notational convenience, we suppress this choice of
the sigma-algebra in our notations.

The adjoint of a linear transformation
$\A \in \L(\hilb_1, \hilb_2)$,
where $\hilb_1$ and
$\hilb_2$ are (real) Hilbert spaces, is denoted by $\A^*$. Recall that
$\A^* \in \L(\hilb_2,\hilb_1)$ and 
\[
   \ipp{\hilb_2}{\A v_1}{v_2} = \ipp{\hilb_1}{v_1}{\A^* v_2}, \quad \text{for all } v_1 
\in \hilb_1, 
 v_2 \in \hilb_2. 
\]
We also recall the following basic result regarding affine transformations of 
Gaussian random variables. Let $X$ be an $\hilb_1$-valued Gaussian random variable 
with law $\GM{a}{\C}$, $\A \in \L(\hilb_1,\hilb_2)$, and $b \in \hilb_2$. Then,
the random variable $\A X + b$ is an $\hilb_2$-valued Gaussian
random variable with law $\GM{\A a + b}{\A \C \A^*}$; see~\cite{DaPrato06} for details.

\subsection{Linear models}\label{sec:linear_additive}
Consider the model
\begin{equation}\label{equ:model_simple}
    \obs = \SS\iparm + \TT\iparb + \vec\eta,
\end{equation}
where $\obs \in \R^d$ denotes measurement data, $\iparm \in \hilbm$ is the inversion parameter,
$\iparb \in \hilbb$ is the secondary uncertain parameter, 
and $\vec\eta$ is the measurement noise vector.
(The spaces $\hilbm$ and $\hilbb$ are as described in the introduction.)
This type of model, which was considered in~\cite{AlexanderianPetraStadlerEtAl21}, 
may correspond to inverse problems governed by linear PDEs with
uncertainties in source terms or boundary conditions. 
We assume
$\SS \in \L(\hilbm, \R^d)$ and 
$\TT \in \L(\hilbb, \R^d)$.
Moreover, 
we assume that $\vec\eta \sim \GM{\vec{0}}{\ncov}$
and that $\vec\eta$ is independent of $\iparm$ and $\iparb$.
We consider a Gaussian prior law $\prior = \GM{\iparprm}{\Cprior}$
for $\iparm$, and 
for the purpose of this illustrative example, 
let the \emph{secondary model uncertainty}
$\iparb$ be distributed according to a Gaussian $\mu_\iparb = \GM{\bar\iparb}{\Caux}$.
Also, we assume $\ncov = \sigma^2 \mat{I}$, with $\sigma^2$ denoting the noise level.

\paragraph{Incorporating model uncertainty in the inverse problem}
For a fixed realization of $\iparb$, estimating $\iparm$ from
\cref{equ:model_simple} is a standard problem. This also follows the 
common
practice of fixing additional model parameters at some nominal values before solving
the inverse problem. However, it is possible to account for the model 
uncertainty in this process. To see this, suppose we fix $\iparb$ at $\bar\iparb$ and consider 
the approximate model $\cfwd(\iparm) = \SS \iparm + \TT \bar\iparb$.
Note that 
\[
\underbrace{\SS\iparm + \TT\iparb}_{\text{accurate model}} = 
\underbrace{\SS\iparm + \TT\bar\iparb}_{\cfwd(\iparm)}  
+ \underbrace{\TT(\iparb-\bar\iparb)}_{\text{error } \vec\eps}.
\]
In this case, the approximation error $\vec\eps = \TT(\iparb-\bar\iparb)$ 
has a Gaussian law, 
$\vec\eps \sim \GM{\vec{0}}{\TT \Caux \TT^*}$. 
Hence, we may rewrite \cref{equ:model_simple} in terms of the approximate model
$\cfwd$ as follows:
\begin{equation}\label{equ:inaccurate_noise_model}
    \obs = \cfwd(\iparm) + \vec\nu,
\end{equation}
where $\vec\nu = \vec\eps + \vec\eta$ denotes the \emph{total error}.
In the present setting, $\vec\nu \sim \GM{\vec 0}{\TT \Caux \TT^* + \ncov}$. Note that we have
incorporated the uncertainty due to $\iparb$ in the error covariance matrix.\footnote[1]{
If the approximate model $\cfwd$ was defined
by fixing $\iparb$ at a point different from $\bar\iparb$, then the error term
$\vec\nu$ would have nonzero mean.} The present procedure for incorporating the
secondary model uncertainty into the inverse problem is a special case of the
Bayesian approximation error (BAE) approach (see~\cref{sec:BAE}).  

Since we have Gaussian prior and noise models and $\cfwd$ 
in~\cref{equ:inaccurate_noise_model} is affine, the posterior is also Gaussian
with analytic formulas for its mean and covariance operator; see
e.g.,~\cite{Stuart10}.  In particular, the posterior covariance operator is
given by 
\begin{equation}\label{equ:Cpost_lin}
   \Cpost =
\big(\SS^*\nucov^{-1}\SS + \Cprior^{-1}\big)^{-1} =  
\big(\SS^*(\TT \Caux \TT^* + \ncov)^{-1}\SS + \Cprior^{-1}\big)^{-1}.
\end{equation}
Considering, for example, the A-optimality criterion $\trace(\Cpost)$,
\cref{equ:Cpost_lin} illustrates the manner in which the uncertainty due to
use of an approximate model impacts the posterior uncertainty.

\paragraph{Interplay between measurement error and approximation error}
Note that $\TT \Caux \TT^* \in \Lsymp(\R^d)$. This operator
admits 
a spectral decomposition 
$\mat{V} \mat{\Lambda} \mat{V}^\tran$, where $\mat{V}$ is an orthogonal matrix
of eigenvectors and $\mat{\Lambda}$ is a diagonal matrix with the 
eigenvalues on it diagonal. Thus, the total error covariance matrix
can be written as 
$ 
\nucov = \mat{V} \mat{\Lambda} \mat{V}^\tran + \sigma^2 \mat{I} 
= \sum_j (\lambda_j + \sigma^2) \vec{v}_j \vec{v}_j^\tran$.
Therefore, the modes for which 
$\lambda_j \ll \sigma^2$ may be ignored.
To observe the impact of 
model uncertainty on individual observations, we note that
\[
\var\{ \nu_i \} = \vec{e}_i^\tran \nucov \vec{e}_i = 
\sum_j (\lambda_j + \sigma^2) (\vec{e}_i^\tran\vec{v}_j)^2, \quad i \in \{1, \ldots, d\},
\]
where $\vec{e}_i$'s are the standard basis vectors in $\R^d$. Thus, we see that 
only large eigenvalues contribute significantly to the total error in the $i$th measurement.

We can also consider the interplay between the spectral representation of
the error covariance and the posterior covariance operator. Specifically, 
$\Cpost =  
\big(\SS^*\nucov^{-1}\SS + \Cprior^{-1}\big)^{-1} = \Cprior^{1/2}
(\tilde\SS^*\nucov^{-1}\tilde\SS + I)^{-1}\Cprior^{1/2}$ with $\tilde\SS = \SS\Cprior^{1/2}$.
Thus, letting $\mat{E}_j = \vec{v}_j \vec{v}_j^\tran$, $j = 1, \ldots, d$, we can write
\[  
\Cpost =  \Cprior^{1/2}
          \Big[\sum_{j=1}^d (\lambda_j + \sigma^2)^{-1} \tilde\SS^* \mat{E}_j \tilde\SS + I\Big]^{-1}
          \Cprior^{1/2}. 
\]
Note that in the directions corresponding to very large eigenvalues the measurements will 
have a negligible impact on posterior uncertainty.

The above discussion indicates that model uncertainty cannot be ignored in the
inverse problem, especially when the model uncertainty overwhelms measurement
noise. The latter is also important in design of experiments.  Specifically,
some of the measurements might be completely useless due to large amount of
model uncertainty associated to the corresponding measurements. This
information needs to be accounted for in the OED problem to ensure only
measurements that are helpful in reducing posterior uncertainty are selected. 

\paragraph{Connection to post-marginalization}
In general, the BAE approach involves pre-marginalization over the secondary model
uncertainties.  This can be related to the idea of post-marginalization in the
case of linear Gaussian inverse problems.  Namely, if we consider $\iparb$ as a
reducible uncertainty that is being estimated along with $\iparm$ and
$\mu_\iparb$ as the corresponding prior law, then \cref{equ:Cpost_lin} is 
the covariance operator of the marginal posterior law of $\iparm$.
To see this, we first note that 
\begin{equation}\label{equ:SMW}
(\TT \Caux \TT^* + \ncov)^{-1} = 
\ncov^{-1} - \ncov^{-1} \TT (\Caux^{-1} + \TT^* \ncov^{-1} \TT)^{-1} \TT^* \ncov^{-1}.
\end{equation}
This relation 
can be derived by following a similar
calculation as the one in \cite[p.~536]{Stuart10}.\footnote[2]{
Note that \cref{equ:SMW} can be viewed as a special form of the Sherman--Morrison--Woodbury 
formula involving Hilbert space operators.} 
Then, we substitute \cref{equ:SMW} in \cref{equ:Cpost_lin} to obtain
\[
    \Cpost = \big[\Cprior^{-1} + \SS^* \ncov^{-1} \SS 
                           - \SS^* \ncov^{-1} \TT (\Caux^{-1} + \TT^* \ncov^{-1} \TT)^{-1} 
                                \TT^* \ncov^{-1}\SS\big]^{-1}.
\]
This $\Cpost$ is the same as the marginal posterior covariance operator
of $\iparm$ as noted in~\cite[Equation 2.4]{AlexanderianPetraStadlerEtAl21}.
Thus, for a linear Gaussian inverse problem with reducible secondary
uncertainty, the posterior covariance operator obtained using
the BAE approach equals the marginal posterior covariance operator of
$\iparm$ obtained following joint estimation of $\iparm$ and $\iparb$.
%
%

\subsection{Nonlinear models}
\label{sec:nonlinear}
The linear model \cref{equ:model_simple} can be  
generalized in the following ways:
\begin{align}
   &\text{additive model, linear in $\iparm$, nonlinear in $\iparb$}:
   \afwd(\iparm, \iparb) = \SS\iparm + \TT(\iparb); \label{equ:ALN}\\
   &\text{additive model, nonlinear in $\iparm$, linear in $\iparb$}:
   \afwd(\iparm, \iparb) = \SS(\iparm) + \TT\iparb; \label{equ:ANL}\\ 
   &\text{additive model, nonlinear in both $\iparm$ and $\iparb$}:
   \afwd(\iparm, \iparb) = \SS(\iparm) + \TT(\iparb); \label{equ:ANN}\\ 
   &\text{nonadditive model, nonlinear in both $\iparm$ and $\iparb$.} \label{equ:NNN}
\end{align}
While the cases \cref{equ:ALN}--\cref{equ:ANN} might be of independent
interest, our focus in this article is on the general case \cref{equ:NNN}. 
However, items \cref{equ:ALN}--\cref{equ:ANN} do serve to illustrate
some of the key challenges.

In the case of \cref{equ:ALN}, one can repeat the steps leading to
\cref{equ:inaccurate_noise_model}, except $\vec\eps$ will not be Gaussian
anymore and therefore the distribution of $\vec\nu$ will not be known analytically.  In that case, one
may obtain a Gaussian approximation to $\vec\eps$ either by fitting a Gaussian
to $\vec\eps$ or by using a linear approximation of $\TT(\iparb)$.  Then, one may
obtain a Gaussian posterior, where one also relies on the linearity of $\SS$.  On
the other hand, in the case of \cref{equ:ANL}, the total error $\vec\nu$ will
be Gaussian as before, but due to nonlinearity of $\SS(\iparm)$ the posterior
will not be Gaussian. The latter leads to one of the fundamental challenges in OED of
nonlinear inverse problem---defining a suitable OED objective whose
optimization is tractable. The more complicated cases of
\cref{equ:ANN}---\cref{equ:NNN} inherit the challenges corresponding to the
previous cases.  

In the rest of this article, we  build on the BAE approach to incorporate the
uncertainty in $\iparb$ in inverse problems governed by nonlinear models of the
type~\cref{equ:NNN}. The uncertainty in $\iparb$ will be assumed irreducible,
and in general, $\iparb$ will not be assumed to follow a Gaussian law.  All
that we require is the ability to generate samples of $\iparb$.  Following the
BAE approach, we approximate the approximation error $\vec\eps$ with a
Gaussian.  This enables incorporating the model uncertainty in the data
likelihood; see \Cref{sec:BAE}.  To cope with non-Gaussianity of
the posterior, we rely on a Gaussian approximation to the posterior, to derive
an uncertainty aware OED objective that is tractable to evaluate and optimize
for infinite-dimensional inverse problems; see \Cref{sec:OED_BAE} for the
definition of the OED objective and \Cref{sec:methods} for computational
methods.

\section{Infinite-dimensional Bayesian inverse problems under uncertainty}\label{sec:BAE}
We consider the inverse problem of inferring a parameter $\iparm$ from a model
of the form~\cref{equ:model}, where $\afwd: \hilbm \times \hilbb \to \R^d$ is a
parameter-to-observable map that in general is nonlinear in both arguments.
We focus on problems
where $\afwd$ is defined as a composition of an  
observation operator and a PDE solution operator.  
The model $\afwd$ is assumed to be Fr\'{e}chet differentiable in $\iparm$, 
at $\iparb = \bar\iparb$, where $\bar\iparb \in \hilbb$ is a nominal value. 
As before, we employ 
a Gaussian noise model, $\vec\eta \sim \GM{\vec{0}}{\ncov}$, 
a Gaussian prior
law $\prior = \GM{\iparprm}{\Cprior}$ for $\iparm$, and assume $\vec\eta$ is 
independent of $\iparm$ and $\iparb$.
The prior 
induces the Cameron--Martin space $\CM = \ran(\Cprior^{1/2})$, which is endowed with the 
inner product~\cite{DaPrato06,DashtiStuart17}
\[
\cip{a}{b} = \mip{\Cprior^{-1/2} a}{\Cprior^{-1/2} b}, \quad a, b \in \CM,
\]
where $\mip{\cdot}{\cdot}$ is the inner product on the parameter space
$\hilbm$. 

To account for model uncertainty (due to $\iparb$) in the inverse problem, we
rely on the BAE approach~\cite{KaipioSomersalo07,KaipioKolehmainen13}, which we
explain next.
We fix the secondary parameter to $\bar\iparb$ and consider 
the approximate (also known as inaccurate, reduced order, or surrogate) model
\begin{equation}\label{equ:model_approx}
    \cfwd(\iparm) = \afwd(\iparm, \bar\iparb).
\end{equation}
As mentioned before, this is typically what is done in practice where the
secondary model parameters are fixed at some (possibly well-justified) nominal values. In the BAE
approach, we quantify and incorporate errors due to the use of this
approximate model in the Bayesian inverse problem analogously to what was done 
in~\Cref{sec:linear_additive}. Namely, we consider
\begin{equation}\label{equ:BAE_model}
\obs = \cfwd(\iparm) + \underbrace{\afwd(\iparm,\iparb)-\cfwd(\iparm)}_{\text{approximation error } 
     \vec\eps} + \vec{\eta} 
     = \cfwd(\iparm) + \vec{\nu}(\iparm,\iparb),
\end{equation}
where $\vec{\nu}(\iparm,\iparb) = \vec\eps(\iparm,\iparb) + \vec\eta$ is the total error.
In the BAE framework, the approximation
error $\vec{\eps}$, is approximated as a conditionally
Gaussian random variable. That is, the distribution of $\vec\eps | \iparm$ is
assumed to be Gaussian.
In the present work, we employ
the so-called {\em enhanced error model}~\cite{KolehmainenTarvainenArridgeEtAl11,KaipioKolehmainen13,
BabaniyiNicholsonVillaEtAl21},
that ignores the correlation between $\vec\eps$ and $\iparm$ and 
approximates the law of $\vec\eps$ as a Gaussian $\vec\eps \sim 
\mathcal{N}(\vec{\eps}_0,\ecov)$ with 
\begin{equation}\label{equ:error_statistics}
\begin{aligned}
    \vec\eps_0 &= \int_\hilbm\int_\hilbb 
           \vec\eps(\iparm,\iparb) \, \prior(d\iparm)\, \mu_\iparb(d\iparb),
\\
    \ecov &= \int_\hilbm\int_\hilbb 
       \big(\vec\eps(\iparm,\iparb) - \vec\eps_0\big)
\big(\vec\eps(\iparm,\iparb) - \vec\eps_0\big)^\tran \, \prior(d\iparm)\, \mu_\iparb(d\iparb).
\quad 
\end{aligned}
\end{equation}
In general, the approximation errors can be (highly) correlated with the parameters~\cite{NicholsonPetraVilla23,KaipioKolehmainen13}. However, ignoring the correlation between $\vec\eps$ and $\iparm$ (i.e., employing the enhanced error model) is typically viewed as a conservative (safe) approximation as it is analogous to approximating the conditional distribution of $\vec\eps | \iparm$ with the marginal distribution of $\vec\eps$, which cannot reduce variance~\cite[Section 3.4]{KaipioSomersalo05}. On the other hand, employing the enhanced error model can significantly reduce the costs associated with computing the mean and covariance operator of $\vec\eps$~\cite{KaipioKolehmainen13} which are in practice computed via sampling;
see \Cref{sec:methods}.

With these approximations, and by our assumption on the measurement noise (which is independent of 
the parameters), the total error $\vec\nu$ is modeled by a Gaussian 
$\GM{\vec\eps_0}{\nucov}$, where $\nucov = \ecov + \ncov$. 
Using this approximate noise model along with 
the approximate model $\cfwd$ we arrive at the following 
data likelihood:
\begin{equation}\label{eq:BAElike}
\like(\obs\vert\iparm) \propto 
\exp\Big\{-\frac{1}{2} \big(\obs - \cfwd(\iparm) - \vec{\eps}_0\big)^\tran
\nucov^{-1}\big(\obs - \cfwd(\iparm) - \vec{\eps}_0\big)\Big\}.
\end{equation}
With the prior measure in place, and using this BAE-based data likelihood, 
we can state the Bayes formula~\cite{Stuart10},
\[
\frac{d\postm}{d\prior} \propto \like(\obs | \iparm).
\]
To make computations involved in design of large-scale inverse 
problems tractable, we rely on a local 
Gaussian approximation to the posterior. Namely, we use $\postmG = \GM{\iparmap^\obs}{\Cpost^\obs}$,
where $\iparmap^\obs$ is the maximum a posteriori probability (MAP)
point and $\Cpost^\obs$ is an approximate posterior 
covariance operator, described below. The MAP point is given by
\begin{equation}\label{equ:map_estimation_problem}
\iparmap^\obs = \argmin_{\iparm\in\mathscr{E}} \mathcal{J}(\iparm)
\defeq  \frac12 \big(\obs - \cfwd(\iparm) - \vec{\eps}_0\big)^\tran\nucov^{-1}
                 \big(\obs - \cfwd(\iparm) - \vec{\eps}_0\big)
       + \frac12 \cip{\iparm - \iparprm}{\iparm - \iparprm}.
\end{equation}
For the approximate posterior covariance operator $\Cpost^\obs$, we use 
\begin{equation}\label{equ:post_cov_GN}
    \Cpost^\obs = \big(\cfwd_m(\iparmap^\obs)^*\nucov^{-1}
    \cfwd_m(\iparmap^\obs) + \Cprior^{-1}\big)^{-1},
\end{equation}
where $\cfwd_m(\iparmap^\obs)$ is the Fr\'{e}chet derivative of $\cfwd$ evaluated at $\iparmap^\obs$.

Note that the true posterior $\postm$ is equivalent to the prior measure.  It is
also possible to show that the Gaussian approximation is equivalent to $\prior$,
as well. This fact, which is made precise below, is important
in justifying the use of this Gaussian approximation for defining an approximate
measure of posterior uncertainty. Namely, to define a notion of uncertainty 
reduction, it is important that our surrogate for the posterior measure is
absolutely continuous with respect to our reference measure, which is given by
the prior.  The equivalence of $\postmG$ to $\prior$ may be inferred from the more general
developments in~\cite{PinskiSimpsonStuartEtAl15}. However, we present an
accessible argument that applies to the specific problem setup under study in
the present work. We first present the following result. 
\begin{proposition} 
\label{prp:flin}
Consider the linearized forward model
\begin{equation}\label{equ:flin}
  \flin(m) = \cfwd(\iparmap) + \cfwd_m(\iparmap)(\iparm - \iparmap), \quad \iparm \in \hilb, 
\end{equation} 
where we have suppressed 
the dependence of $\iparmap$ to $\obs$ for notational convenience. 
Define the data model 
\begin{equation}\label{equ:linear_data_model}
  \obs = \flin(\iparm) + \vec\nu,
\end{equation}  
where $\vec\nu \sim \GM{\vec{\eps}_0}{\nucov}$. Consider 
the Bayesian inverse problem of 
estimating $\iparm$ using~\cref{equ:linear_data_model} and  
the prior $\prior = \GM{\iparprm}{\Cprior}$.
The corresponding posterior measure is given 
by $\postmG$.
\end{proposition}

\begin{proof}
Using the theory of Bayesian linear inverse problems in 
a Hilbert space~\cite{Stuart10}, the solution of the linear inverse problem 
under study 
yields a Gaussian posterior $\postmSimp^{\text{lin}} = \GM{\iparmap^\text{lin}}{\Cpost}$.
The covariance operator $\Cpost$ is as in~\cref{equ:post_cov_GN} and 
$\iparmap^{\text{lin}}$ is found by minimizing 
\begin{equation}
\JL(\iparm)
\defeq  \frac12 \big(\obs - \flin(\iparm) - \vec{\eps}_0\big)^\tran\nucov^{-1}
                 \big(\obs - \flin(\iparm) - \vec{\eps}_0\big)
       + \frac12 \cip{\iparm - \iparprm}{\iparm - \iparprm},
\end{equation}
over the Cameron--Martin space $\CM$. To complete the proof, we 
show $\iparmap^{\text{lin}} = \iparmap$.
Note that $\JL$ is a strictly convex quadratic functional with a unique global
minimizer. We consider the Euler--Lagrange equation for the present optimization problem. 
Note that the Fr\'{e}chet derivative of $\flin$ is given by 
$\flin_m = \cfwd_m(\iparmap)$.  
It is straightforward to see 
\[
  \frac{d}{d\eps}\Big|_{\eps = 0} \JL(m + \eps \ut{m}) 
  = \mip{\ut{m}}{\cfwd_m(\iparmap)^*\nucov^{-1}(\vrl)} + \cip{\ut{m}}{\iparm - \iparprm},
\]
for every $\ut{m} \in \CM$.
Thus, 
recalling $\flin(\iparmap) = \cfwd(\iparmap)$, we note that 
for every $\ut{m} \in \CM$,
\[
\begin{aligned}
  \frac{d}{d\eps}\Big|_{\eps = 0} \JL(\iparmap + \eps \ut{m}) 
  &= \mip{\ut{m}}{\cfwd_m(\iparmap)^*\nucov^{-1}(\vrfmap)} + \cip{\ut{m}}{\iparmap - \iparprm}
  \\
  &= \frac{d}{d\eps}\Big|_{\eps = 0} \mathcal{J}(\iparmap + \eps\ut{m}) = 0.
\end{aligned}
\]
The last equality follows from the fact that $\iparmap$ is a minimizer of $\mathcal{J}$
in~\cref{equ:map_estimation_problem}. Hence, $\iparmap$ is the unqiue global minimizer of 
$\JL$. Therefore, $\iparmap^{\text{lin}} = \iparmap$. 
\end{proof}
\Cref{prp:flin} shows that $\postmG$ is the posterior measure 
corresponding to the linearized Bayesian inverse problem considered 
in the result. Therefore, 
by construction, $\postmG$ is equivalent to $\prior$.

\section{A-optimal experimental design under uncertainty}\label{sec:OED_BAE}
We consider inverse problems in which measurement data
are collected at a set of sensors.
In this case, the OED problem seeks to find an
optimal placement of sensors. Specifically, we formulate the OED problem as
that of selecting an optimal subset from a set of candidate sensor locations,
which is a common approach; see e.g.,~\cite{Ucinski05,
HaberHoreshTenorio08,AlexanderianPetraStadlerEtAl14}.
To make matters concrete, we begin by fixing a set of points $\{ \vec{x}_1, \vec{x}_2, 
\ldots, \vec{x}_\Ns\}$ that indicate the candidate sensor locations.  We then assign a
binary weight $w_i$ to each candidate location $\vec{x}_i$; a weight of one
indicates that a sensor will be placed at the corresponding candidate
location.  The binary vector $\vec w \in \{0, 1\}^\Ns$ thus fully specifies an
experimental design in the present setting.
Note that specification of a set of candidate sensor locations will in general 
depend on the specific application at hand. For example, placing sensors in 
certain parts of the domain might be impractical or impossible. Also, in
problems with Dirichlet boundary conditions, placing sensors on or very close to
such boundaries will be a waste of resources. 

\subsection{Design of the Bayesian inverse problem} The design $\vec w$ enters
the formulation of the Bayesian inverse problem through the data likelihood.
This requires additional care in the present work because the total error
covariance matrix $\nucov$ is non-diagonal.  We follow the setup
in~\cite{LiuChepuriFardad16} to incorporate $\vec w$ in the Bayesian inverse
problem formulation.
For a binary design vector $\vec{w} \in \{0, 1\}^\Ns$, we define the matrix
$\PPsi$ as submatrix of $\mathrm{diag}(\vec w)$ with rows corresponding to the
zero weights removed. Thus, given a generic measurement vector $\vec{d} \in
\R^\Ns$, $\PPsi \vec{d}$ returns the measurements corresponding to active
sensors. For $\vec{d} \in \R^\Ns$, we use the notation $\vec{d}_\vec{w} = \PPsi
\vec{d}$.

Next, we describe how a design vector $\vec w$ enters the Bayesian inverse
problem, within the BAE framework. 
For a given $\vec w$, we consider the model $\yhat = \PPsi (\cfwd(\iparm) + \vec\nu)$. 
Using this model leads to the following, $\vec w$-dependent, 
data likelihood
\begin{equation}\label{eq:BAElikeW}
\like(\obs\vert\iparm) \propto
\exp\Big\{-\frac{1}{2} \big(\obs - \cfwd(\iparm) - \vec{\eps}_0\big)^\tran
\WW(\vec w)
\big(\obs - \cfwd(\iparm) - \vec{\eps}_0\big)\Big\},
\end{equation}
with 
\begin{equation}\label{equ:Sigma}
\WW(\vec w) = \PPsi^\tran \nucovw^{-1} \PPsi, \quad \text{where} \quad
\nucovw =\PPsi\nucov\PPsi^\tran.
\end{equation}

Consequently, we obtain the following $\vec w$-dependent Gaussian approximation to 
the posterior, $\postmGw = \GM{\iparmap^\yhat}{\Cpost^\yhat}$, 
where $\iparmap^\yhat$ is obtained by minimizing 
\begin{equation}\label{equ:costw}
\mathcal{J}_{\vec{w}}(\iparm; \obs) =
         \frac12 \big(\obs - \cfwd(\iparm) - \vec{\eps}_0\big)^\tran
                 \WW(\vec w) 
                 \big(\obs - \cfwd(\iparm) - \vec{\eps}_0\big)
       + \frac12 \cip{\iparm - \iparprm}{\iparm - \iparprm},
\end{equation}
and 
\begin{equation}\label{equ:Cpostw}
 \Cpost^\yhat = \big(\cfwd_m(\iparmap^\yhat)^*
    \WW(\vec w) 
    \cfwd_m(\iparmap^\yhat) + \Cprior^{-1}\big)^{-1}.
\end{equation}

Note that in practical computations, 
the cost function $\mathcal{J}_{\vec w}$ in~\cref{equ:costw} is 
implemented as  
\[
\mathcal{J}_\vec{w}(\iparm; \obs) =
         \frac12 \big(\yhat - \PPsi\cfwd(\iparm) - \PPsi\vec{\eps}_0\big)^\tran 
                 \nucovw^{-1} 
                 \big(\yhat - \PPsi\cfwd(\iparm) - \PPsi\vec{\eps}_0\big)
       + \frac12 \cip{\iparm - \iparprm}{\iparm - \iparprm},
\]
with $\nucovw$ as in~\cref{equ:Sigma}.
This amounts to using the data from the active sensors and removing the
rows and columns corresponding to inactive sensors from
the error covariance matrix $\nucov$.

\subsection{The design criterion}
In the present work, we follow an A-optimal design strategy, where the goal is
to find designs that minimize the average posterior variance.  Generally,
computing the average posterior variance for a Bayesian nonlinear inverse
problem is computationally challenging.  We follow the developments
in~\cite{AlexanderianPetraStadlerEtAl16} to define a Bayesian A-optimality
criterion in the case of nonlinear inverse problems.  

Given a data vector $\vec y \in \R^\Ns$, an approximate measure of posterior uncertainty
is provided by $\trace(\Cpost^\yhat)$. However, when solving the OED problem data is not available.
Indeed, it is the goal of the OED problem to specify how data should be collected. To overcome this,
we follow the general approach in Bayesian OED of nonlinear inverse problems, where we consider  
$\mathbb{E}_\obs \{ \trace(\Cpost^\yhat) \}$,
with $\mathbb{E}_\obs$ denoting expectation with respect to the set of all likely data. 
We compute this 
expectation by using the information available in the Bayesian inverse problem and the 
information regarding the distribution of the model uncertainty.
Namely, following the approach in~\cite{AlexanderianPetraStadlerEtAl16}, we use the 
design criterion
\begin{equation}\label{equ:design_criterion}
\Phi(\vec w) = \int_\hilbb 
\int_\hilbm \int_{\R^\Ns} \trace(\Cpost^\yhat) \, \noise( \afwd(\iparm, \iparb) - \obs)d\obs 
                         \, \prior(dm)\, \mu_\iparb(d\iparb),
\end{equation}
where $\noise$ is the probability density function (pdf) of the noise distribution 
$\GM{\vec{0}}{\ncov}$.
In practice, $\Phi(\vec w)$ will be computed via sample
averaging. Specifically, we use 
\begin{equation}\label{equ:design_criterion_mc}
   \Phi_\Nd(\vec w) = \frac{1}{\Nd} \sum_{i=1}^\Nd \trace(\Cpost^{\yhati}),
\end{equation}
where the \emph{training} data samples $\yhati$ are given by 
$\yhati = \PPsi(\afwd(\iparm^i,\iparb^i) + \vec\eta^i)$,
with $\{ (\iparm^i, \iparb^i, \vec\eta^i) \}_{i=1}^\Nd$
a sample set from the product space $(\hilbm, \prior) \otimes (\hilbb,
\mu_\iparb) \otimes (\R^\Ns, \GM{\vec{0}}{\ncov})$.  
In large-scale
applications, typically a small $\Nd$ can be afforded. However, in this
context, typically a modest $\Nd$ enables computing good quality optimal
designs. This is also demonstrated in the computational results in the present
work.

Note that, for each $i \in \{1, \ldots, \Nd\}$, 
\[
\trace(\Cpost^{\yhati}) = \sum_{j=1}^\infty \mip{\Cpost^{\yhati}e_j}{e_j},
\] 
where $\{ e_j \}_{j=1}^\infty$ is a complete orthonormal set in $\hilbm$. Inserting 
this in~\cref{equ:design_criterion_mc} and using the definition of $\Cpost^{\yhati}$, 
we have
\begin{subequations}
\label{equ:design_criterion_mc_infsum_def}
\begin{equation}
   \label{equ:design_criterion_mc_infsum}
   \Phi_\Nd(\vec w) = \frac{1}{\Nd} \sum_{i=1}^\Nd \sum_{j=1}^\infty 
   \mip{c_{ij}}{e_j},
\end{equation}
   where, for $i \in \{1, \ldots, \Nd\}$ and $j \in \mathbb{N}$,
\begin{align}
   &\iparmap^\yhati= \argmin_\iparm \mathcal{J}_{\vec{w}}(\iparm; \obs^i),
   \label{equ:map_estimation}
\\
   &\big(\cfwd_m(\iparmap^\yhat)^*
    \WW(\vec w) \cfwd_m(\iparmap^\yhat) + \Cprior^{-1}\big)c_{ij} = e_j.
   \label{equ:hess_sys}
\end{align}
\end{subequations}

Computing the OED objective as defined above is not practical.  However,
\cref{equ:design_criterion_mc_infsum_def} provides insight into the key challenges
in computing the OED objective. In the first place, \cref{equ:map_estimation}
is a challenging PDE-constrained optimization problem whose solution is the MAP point
$\iparmap^\yhati$.  Moreover, upon discretization, \cref{equ:hess_sys} will be
a high-dimensional linear system with the discretization of the operator
$\big(\cfwd_m(\iparmap^\yhat)^* \WW(\vec w) \cfwd_m(\iparmap^\yhat) + \Cprior^{-1}\big)$ as its coefficient matrix.  Such
systems can be tackled with Krylov iterative methods that require the
application of the coefficient matrix on vectors.  In \Cref{sec:methods}, we
present two approaches for efficiently computing the OED objective: one
approach uses randomized trace estimation and the other utilizes low-rank
approximation of $\cfwd_m(\iparmap^\yhat)^* \WW(\vec w) 
\cfwd_m(\iparmap^\yhat)$. These approaches rely on scalable optimization
methods for \cref{equ:map_estimation} as well as adjoint based gradient and
Hessian apply computation.

\subsection{The optimization problem for finding an A-optimal design}
We state the OED problem of selecting the
best $K$ sensors, with $K \leq \Ns$, as follows:
\begin{equation}\label{equ:optimOED}
    \begin{aligned}
    &\min_{\vec{w} \in \{0, 1\}^\Ns} \Phi_\Nd(\vec w) \\
    &\quad\text{s.t.\,} \sum_{\ell=1}^\Ns w_\ell = K.
    \end{aligned}
\end{equation}
This is a challenging binary optimization problem. One possibility 
is to pursue a relaxation strategy~\cite{LiuChepuriFardad16,AlexanderianPetraStadlerEtAl14,
AlexanderianPetraStadlerEtAl16} to enable
gradient-based 
optimization with design weights $w_i \in [0, 1]$.
A practical alternative is to follow a greedy approach to find an 
approximate solution to \cref{equ:optimOED}. Namely, we place  
sensors one at a time. In each step of a greedy algorithm, we select
the sensor that provides the largest decrease in the value of the OED objective
$\Phi_\Nd$.
While the solutions obtained using a greedy algorithm are suboptimal in
general, greedy approaches have shown good performance in many sensor placement
problems; see,
e.g.,~\cite{KrauseSinghGuestrin08,ShulkindHoreshAvron18,
JagalurMarzouk21,AlexanderianPetraStadlerEtAl21}.  In the present work,
we follow a greedy approach for finding approximate solutions
for~\cref{equ:optimOED}. The effectiveness of this approach is demonstrated in
our computational results in~\Cref{sec:numerics}.

\section{Computational methods}\label{sec:methods}
We begin this section with a brief discussion on computing the 
BAE error statistics in~\Cref{sec:baestats}.  We then detail our proposed
methods for computing the OED objective in~\Cref{sec:oed_obj_compute}.  The
greedy approach for computing optimal designs is outlined in~\Cref{sec:greedy}.
Then, we discuss the computational cost of OED objective evaluation, using our
proposed methods, as well as the greedy procedure in~\Cref{sec:cost}.

\subsection{Estimating approximation error statistics}\label{sec:baestats}
As discussed in \Cref{sec:BAE}, the mean and covariance of 
the approximation error in \cref{equ:error_statistics} will be 
approximated via Monte Carlo sampling. Specifically, we begin by drawing  
samples $\{ (\iparm^i, \iparb^i) \}_{i=1}^\Ns$ 
in $(\hilbm \times \hilbb, \prior \otimes \mu_\iparb)$ 
and compute 
\begin{equation}\label{equ:BAE_comp}
        \widehat{\vec\eps}_0 = \frac1\Nsamp \sum_{i=1}^\Nsamp\vec\eps^i, \quad
        \ecovhat = \frac{1}{\Nsamp-1} \sum_{i=1}^\Nsamp (\vec\eps^i - \widehat{\vec\eps}_0)
                                                         (\vec\eps^i - \widehat{\vec\eps}_0)^\tran,
\quad \text{where } 
    \vec\eps^i = 
      \afwd(\iparm^i, \iparb^i) - \cfwd(\iparm^i).
\end{equation}
Subsequently, we 
use $\vec\eps_0 \approx \widehat{\vec\eps}_0$ and $\ecov \approx
\ecovhat$.  The computational cost of this process is $2\Nsamp$ model
evaluations. Note, however, that these model evaluations are done a
priori, in parallel, and can be used for both the OED problem and solving the
inverse problem.  Typically, only a modest sample size is
sufficient for approximating the mean and covariance of the 
model error~\cite{NicholsonPetraKaipio18,BabaniyiNicholsonVillaEtAl21}. 
Specifically, as noted in \Cref{sec:motivation},
only dominant modes of the error covariance matrix need to be
resolved. 
%
%
Note also that (a subset of) the model evaluations in \cref{equ:BAE_comp}
may be reused in generation of training data needed for computation of the
OED objective.

\subsection{Computation of the OED objective}\label{sec:oed_obj_compute}
In this section, we present scalable computational methods for computing the
OED objective. Specifically, we present two methods: (i) a method based on the
use of randomized trace estimators (\Cref{sec:trace_est}) and (ii) a method
based on low-rank spectral decompositions (\Cref{sec:lowrank}).  As discussed
further below, the latter is particularly suited to our overall approach in the
present work.  Therefore, we primarily focus on the method based on low-rank
spectral decompositions, which we also fully elaborate in the context of a
model inverse problem (see \Cref{sec:model} and \Cref{sec:numerics}).  However,
the first method does have its own merits and is included to provide an
alternative approach.  The relative benefits and computational complexity of
these methods are discussed in \Cref{sec:cost}.

Before we proceed further, we define some notations that simplify 
the discussions that follow. 
For $i \in \{1, \ldots, \Nd\}$,
we define 
\begin{equation}\label{equ:Hi}
\H^i(\vec w) = \cfwd_m(\iparmap^\yhati)^*\WW(\vec w)\cfwd_m(\iparmap^\yhati)
\quad \text{and}\quad
\HT^i(\vec w) = \Cprior^{1/2}\H^i(\vec w)\Cprior^{1/2}. 
\end{equation}
Note that 
the operator $\H^i$ is the so-called Gauss--Newton Hessian of the data mistfit
term in the definition of $\mathcal{J}_{\vec{w}}(\iparm;\obs^i)$ in~\cref{equ:costw}, 
evaluated at $\iparmap^\yhati$.

\subsubsection{Monte Carlo trace estimator approach}
\label{sec:trace_est}
We begin by deriving a randomized (Monte Carlo) trace estimator for the
posterior covariance operator~\cref{equ:Cpostw}.  This is facilitated by the
following technical result.

\begin{proposition}
\label{prp:trace}
Let $\C \in \Ltsympp(\hilbm)$ and 
$\mathcal{K} \in \L(\hilbm)$, and 
consider the Gaussian measure $\mu = \GM{0}{\C}$ on $\hilbm$. Then, 
$\int_\hilbm \mip{\K z}{z}\, \mu(dz) = \trace(\C^{1/2} \K \C^{1/2})$.
\end{proposition}
\begin{proof}
By the assumptions on $\C$ and $\K$, we have that
both $\C\K$ and $\C^{1/2} \K \C^{1/2}$
are trace-class. 
Also, by the formula 
for the expectation of 
a quadratic form in the infinite-dimensional Hilbert space setting~\cite[Lemma~1]{AlexanderianGloorGhattas16}, 
we have that $\int_\hilbm \mip{\K z}{z}\, \mu(dz) = \trace(\C\K)$. To finish the proof, we let
$\{e_j\}_{j=1}^\infty$ be the (orthonormal) basis of the eigenvectors of $\C$ with 
corresponding (real, positive) eigenvalues $\{ \lambda_j \}_{j=1}^\infty$,
and note
$
    \trace(\C\K) = \sum_j \mip{\C\K e_j}{e_j} 
    = \sum_j  \lambda_j \mip{\K e_j}{e_j} 
    = \sum_j \mip{\K \C^{1/2}e_j}{\C^{1/2} e_j} 
    = \sum_j \mip{\C^{1/2}\K \C^{1/2}e_j}{e_j}
    = \trace(\C^{1/2}\K \C^{1/2}). 
$
\end{proof}

Next, note that we can write $\Cpost^\yhati$ as 
\begin{equation}
\Cpost^\yhati = \big(\cfwd_m(\iparmap^\yhati)^*\WW(\vec w)\cfwd_m(\iparmap^\yhati) + \Cprior^{-1}\big)^{-1}
=
\Cprior^{1/2} 
 \big(\HT^i(\vec w) + I\big)^{-1} \Cprior^{1/2}.
\end{equation}
Using this and letting $\C = \Cprior$ and $\K = \big(\HT^i(\vec w) + I\big)^{-1}$ 
in~\Cref{prp:trace}, we have 
$\trace(\Cpost^\yhati) = \int_\hilbm \mip{\big(\HT^i(\vec w)+I\big)^{-1} z}{z}\, \mu(dz)$.
Approximating the integral on the right hand side
via sampling, we obtain 
\[
\trace(\Cpost^\yhati) \approx \frac1\Ntr\sum_{j=1}^\Ntr 
\mip{\big(\HT^i(\vec w) + I\big)^{-1} z_j}{z_j},
\]
where $\{ z_j \}_{j=1}^\Ntr$ are draws from the measure $\mu = \GM{0}{\Cprior}$.
This randomized (Monte Carlo) trace estimator allows 
approximating the OED objective~\cref{equ:design_criterion_mc}, 
as follows.
\begin{subequations}
\label{equ:design_criterion_mc_trace_est_def}
\begin{equation}
\label{equ:design_criterion_mc_trace_est}
\PhiNdTR(\vec w) = \frac{1}{\Ntr\Nd} \sum_{i=1}^\Nd \sum_{j=1}^{\Ntr}
   \mip{c_{ij}}{z_j},
\end{equation}
where, for $i \in \{1, \ldots, \Nd\}$ and $j \in \{1, \ldots, \Ntr\}$,
\begin{align}
   &\iparmap^\yhati = \argmin_\iparm \mathcal{J}_{\vec{w}}(\iparm; \obs^i),
   \label{equ:map_estimation_trace_est}
\\
   &\big(\HT^i(\vec w) 
    + I\big)c_{ij} = z_j.
   \label{equ:hess_sys_trace_est}
\end{align}
\end{subequations}

The major computational challenge in computing $\PhiNdTR(\vec w)$ is solving
the MAP estimation problems in~\cref{equ:map_estimation_trace_est}.
These \emph{inner} optimization problems can be solved efficiently using an
inexact Newton-Conjugate Gradient (Newton-CG) method. The required first and second order
derivatives can be obtained efficiently using adjoint-based gradient and
Hessian apply computation.  The Hessian solves in \cref{equ:hess_sys_trace_est}
are also done using (preconditioned) CG. This only requires the action of
$\HT^i(\vec w)$ on vectors, which can be done using the adjoint method. 
A detailed study of the computational cost of computing $\PhiNdTR(\vec w)$ 
is provided in~\Cref{sec:cost}.

\subsubsection{Low-rank approximation approach}\label{sec:lowrank}
Due to the use of finite-dimensional observations, $\HT^i$ in~\cref{equ:Hi} has
a finite-dimensional range. Also, often $\HT^i$ exhibits rapid spectral decay and, in
cases where the dimension of measurements is high, the numerical rank of this
operator is typically much smaller than its exact rank. Note that the exact
rank is bounded by the measurement dimension.  
Such low-rank structures are due to possible smoothing properties of the
forward operator and that of $\Cprior$; see,
e.g.,~\cite{Bui-ThanhGhattasMartinEtAl13}.
The following technical result facilitates exploiting such low-rank structures
to compute the OED objective.

\begin{proposition}\label{prp:lowrank}
Let $\C$ and $\A$ be in $\Ltsymp(\hilbm)$. 
Letting $\{v_k\}_{k=1}^\infty$ be the orthonormal basis of eigenvectors of $\A$ with corresponding 
(real non-negative) eigenvalues $\{\lambda_k\}_{k=1}^\infty$, we have
\begin{equation}\label{equ:prp_lowrank}
\trace(\C^{1/2} (I + \A)^{-1}\C^{1/2}) = \trace(\C) - \sum_{k=1}^\infty \frac{\lambda_k}{1+\lambda_k}\|\C^{1/2} v_k\|_\hilbm^2.
\end{equation}
\end{proposition}
\begin{proof}
The result follows by noting that
\[
\begin{aligned}
\trace(\C) - \trace(\C^{1/2} (I + \A)^{-1}\C^{1/2}) 
       &= \sum_{k=1}^\infty \mip{\C v_k}{v_k} - \sum_{k=1}^\infty \mip{\C (I + \A)^{-1} v_k}{v_k}
\\
       &= \sum_{k=1}^\infty \mip{\C v_k}{v_k} - \sum_{k=1}^\infty \frac{1}{1 + \lambda_k} \mip{\C v_k}{v_k} 
\\
       &= \sum_{k=1}^\infty \frac{\lambda_k}{1+\lambda_k} \mip{\Cprior v_k}{v_k}
       = \sum_{k=1}^\infty \frac{\lambda_k}{1+\lambda_k} \|\C^{1/2} v_k\|_\hilbm^2.
\end{aligned}
\]
\end{proof}
The relation~\cref{equ:prp_lowrank} facilitates approximating 
$\trace(\C^{1/2}(I + \A)^{-1}\C^{1/2})$. Namely, we can truncate the 
infinite summation in the right-hand side to the first $r$ terms, corresponding 
to the $r$ dominant eigenvalues of $\A$. We can also quantify 
the approximation error due to this truncation as follows.
\begin{proposition}\label{prp:lowrank_err}
Let $\A$ and $\C$ be as in~\cref{prp:lowrank} and 
let $\{\lambda_k\}_{k=1}^{r+1}$ be the $r+1$ largest eigenvalues of $\A$. 
Define 
$T_r(\A, \C) =  
\trace(\C) - \sum_{k=1}^r \frac{\lambda_k}{1+\lambda_k}\|\C^{1/2} v_k\|_\hilbm^2$.
Then, 
\begin{equation}\label{equ:prp_lowrank_err}
     |\trace(C^{1/2}(I + \A)^{-1}\C^{1/2}) - T_r(\A, \C)| \leq 
     \frac{\lambda_{r+1}}{1+\lambda_{r+1}} \trace(\C).
\end{equation}
\end{proposition}
\begin{proof}
We have 
\[
\begin{aligned}
  |\trace(C^{1/2}(I + \A)^{-1}\C^{1/2}) - T_r(\A, \C)| 
  &= \sum_{k=r+1}^\infty  \frac{\lambda_k}{1+\lambda_k}\|\C^{1/2} v_k\|_\hilbm^2 \\
  &\leq \frac{\lambda_{r+1}}{1+\lambda_{r+1}} \sum_{k=r+1}^\infty \|\C^{1/2} v_k\|_\hilbm^2 
  \\
  &= \frac{\lambda_{r+1}}{1+\lambda_{r+1}} \sum_{k=1}^\infty \mip{\C v_k}{v_k}
  = \frac{\lambda_{r+1}}{1+\lambda_{r+1}} \trace(\C). \qedhere 
\end{aligned}  
\]
\end{proof}

Next, we consider $\HT^i(\vec w)$ and let
$\{(\lambda_{ik}, v_{ik})\}_{k=1}^r$ be its dominant eigenpairs. 
(The dependence of eigenvalues and eigenvectors on $\vec w$ is
suppressed, for notational convenience.)  Using~\Cref{prp:lowrank} with $\C =
\Cprior$ and $\A = \HT^i(\vec w)$, we obtain the approximation
\begin{equation}\label{equ:lowrank_update}
\begin{aligned}
\trace(\Cpost^\yhati) &= 
      \trace\big(\Cprior^{1/2} (I + \HT^i(\vec w))^{-1}  \Cprior^{1/2}\big) \\
      &= \trace\big(\Cprior(I + \HT^i(\vec w))^{-1}\big) \approx \trace(\Cprior) - 
      \sum_{k=1}^r \frac{\lambda_{ik}}{1+\lambda_{ik}}\| \Cprior^{1/2} v_{ik}\|_\hilbm^2.
\end{aligned}
\end{equation}
Observe that only the second term in the right-hand side depends on $\vec w$.
This term, which provides a measure of uncertainty reduction,
can be used to define the OED objective. Specifically, using~\cref{equ:design_criterion_mc}
along with~\eqref{equ:lowrank_update}, 
leads to following form of the OED objective:
\begin{subequations}
\label{equ:design_criterion_LR_def}
\begin{equation}
\label{equ:design_criterion_LR}
   \PhiNdLR(\vec w) = -\frac{1}{\Nd}\sum_{i=1}^\Nd
   \sum_{k=1}^r \frac{\lambda_{ik}}{1+\lambda_{ik}} \| \Cprior^{1/2} v_{ik}\|^2_\hilbm,
\end{equation}
where, for $i \in \{1, \ldots, \Nd\}$, 
\begin{alignat}{2}
   &\iparmap^\yhati= \argmin_\iparm \mathcal{J}_{\vec{w}}(\iparm; \obs^i),
   \label{equ:map_estimation_LR}
\\
&\HT^i(\vec w) v_{ik} = 
\lambda_{ik} v_{ik}
&&\quad k \in \{1, \ldots, r\},
\label{equ:eig_LR}\\
&\mip{v_{ik}}{v_{il}} = \delta_{kl}
&&\quad k, l \in \{1, \ldots, r\}.
\label{equ:eig_orthnorm}
\end{alignat}
\end{subequations}
Note that the numerical rank of $\HT^i$ will, in general, depend on $i$. In 
\cref{equ:design_criterion_LR_def}, we have used a common target rank $r$ for
simplicity. In the present setting, this target rank is bounded by the number
of active sensors for a given $\vec w$. Note also that~\Cref{prp:lowrank_err}
shows how to quantify the error due to the use of 
truncated spectral decompositions. Namely, in view of~\cref{equ:prp_lowrank_err},
the error due to the truncation is 
bounded by $\left(\frac{1}{\Nd}\sum_{i=1}^\Nd\frac{\lambda_{i,r+1}}{1+\lambda_{i,r+1}}\right)\trace(\Cprior)$.


As in the case of the approach outlined in \Cref{sec:trace_est}, the dominant
computational challenge in computing $\PhiNdLR(\vec w)$ is the solution of the
MAP estimation problems~\cref{equ:map_estimation_LR}. This will be tackled
using the same techniques. The eigenvalue problem~\cref{equ:eig_LR} can be
tackled via Lanczos or randomized approaches.  The target rank $r$, can be
selected as the number of active sensors. This is suitable in cases where we
have a small number of active sensors. See~\Cref{sec:cost} for further details
regarding the computational cost of evaluating $\PhiNdLR$.

We point out that if the dominant eigenvalues are simple, the
condition $\mip{v_{ik}}{v_{ik}} = 1$ for all $k \in \{1, \ldots, r\}$, implies 
the orthonormality condition \cref{equ:eig_orthnorm}. This follows from
the fact that the eigenvectors corresponding to distinct eigenvalues of $\HT^i$,
which belongs to $\Ltsymp(\hilbm)$, are orthogonal. The simplicity
assumption on the dominant eigenvalues is observed in many 
applications where the dominant eigenvalues decay rapidly.
Making this simplicity assumption, and letting $s_{ik} = \Cprior^{1/2} v_{ik}$,
we may also write the eigenvalue problem above as the following generalized eigenvalue problem 
\begin{equation}\label{equ:eig_gen}
\begin{aligned}
&\H^i(\vec w) s_{ik} = \lambda_{ik} \Cprior^{-1} s_{ik},
\\
&\cip{s_{ik}}{s_{ik}} = 1,
\end{aligned}
\end{equation}
where $i \in \{1, \ldots, \Nd\}$ and $k\in \{1, \ldots, r\}$.
This formulation is helpful when describing the eigenvalue problem in the weak form, 
as seen in~\Cref{sec:oed-objective}. In particular, the eigenvalue problem~\eqref{equ:eig_gen}
will be formulated in terms of the incremental state and adjoint equations and the adjoint
based expression for $\HT^i$ applies.


\subsection{Greedy optimization}\label{sec:greedy}
We follow a greedy approach for solving the OED problem.
That is, we select sensors one at a time: at each step, we pick the sensor
that results in the greatest decrease in the OED objective value; see 
\cref{alg:greedy}. As mentioned before, we use the 
approach in \Cref{sec:lowrank} for computing the OED objective.
That is, we use the OED objective $\PhiNdLR$, as defined 
in~\cref{equ:design_criterion_LR_def}.
\begin{algorithm}[!ht]
\begin{algorithmic}[1]
\REQUIRE The target number of sensors $K$ in \cref{equ:optimOED}. 
\ENSURE The optimal design vector $\vec w$.
\STATE $\vec{w} \gets  \vec{0}$  
\STATE $\mathcal{I}_\mathrm{candidate} \gets  \{1, \ldots, \Ns\}$
\STATE $\mathcal{I}_{\mathrm{active}} \gets \emptyset$ 
\FOR{$\ell = 1$ \TO $K$} 
\STATE Evaluate $\PhiNdLR(\vec{w} + \vec{e}_j)$, for all $j \in \mathcal{I}_\mathrm{candidate}$
\hfill
\COMMENT{\small$\vec{e}_j$ is the $j$th coordinate vector in $\R^\Ns$}
\STATE $\displaystyle i_\ell \gets \displaystyle  \argmin_{j \in \mathcal{I}_\mathrm{candidate}} 
        \PhiNdLR(\vec{w} + \vec{e}_j)$
\STATE $\mathcal{I}_\mathrm{active} \gets \mathcal{I}_\mathrm{active} \cup \{ i_\ell \}$
\STATE $\mathcal{I}_\mathrm{candidate} \gets \mathcal{I}_\mathrm{candidate} \setminus \{ i_\ell \}$ 
\STATE $\vec{w} \gets \vec{w} + \vec{e}_{i_\ell}$
\ENDFOR
\end{algorithmic}
\caption{Greedy approach for solving the OED problem~\cref{equ:optimOED} with 
$\Phi_\Nd = \PhiNdLR$.} 
\label{alg:greedy}
\end{algorithm}

Theoretical justifications behind the use of a greedy approach for sensor
placement, in various contexts, have been investigated in a number of
works~\cite{KrauseSinghGuestrin08,
ShulkindHoreshAvron18,JagalurMarzouk21}. The solution obtained using the greedy
algorithm is known to be suboptimal.  However, as observed in practice, 
the use of a greedy algorithm is a
practical approach and is effective in obtaining near
optimal sensor placements; see
also~\cite{Li19,WuChenGhattas23,AlexanderianPetraStadlerEtAl21}.  We
demonstrate the effectiveness of this approach, in our computational results in
\Cref{sec:numerics}.  

\subsection{Computational cost of sensor placement}\label{sec:cost}
The $\ell$th step of the greedy algorithm requires $\Ns - \ell - 1$ 
OED objective evaluations (cf.\ step 5 of \cref{alg:greedy}); these can be performed in parallel.
It is straightforward to note that 
placing $K$ sensors using the greedy approach requires a total of $C(K, \Ns) := K \Ns -
K(K-1)/2$ OED objective function evaluations.  
Therefore, the overall cost of computing an optimal design, in terms of the
number of PDE solves, using the proposed approach is $C(K, \Ns)$ times the
number of PDE solves required in each OED objective evaluation. Thus, to
provide a complete picture, in this section, we detail the cost of OED objective evaluation
using the two approaches discussed in \Cref{sec:oed_obj_compute}.
A key aspect of both of these approaches is that the cost of OED 
objective evaluation, in terms of the number of required PDE solves,
is independent of the dimension of the 
discretized inversion and secondary parameters.

In the following discussion, the rank of the operators $\HT^i(\vec w)$ in
\cref{equ:Hi} plays an important role. As mentioned before, the exact rank of
these operators is bounded by the number of active sensors. Moreover, if the
number of active sensors is high, the numerical rank is typically considerably
smaller than the exact rank. We denote the number of active sensors in a given
design by $\Nact = \Nact(\vec w)$.  Note that for a binary design vectors $\vec
w$, $\Nact(\vec w) = \| \vec w\|_1$.

\paragraph{Cost of evaluating $\PhiNdTR$}
The most expensive step in evaluating
\cref{equ:design_criterion_mc_trace_est_def} is solving the inner optimization
problems for the MAP points $\iparmap^\yhati$, $i \in \{1, \ldots, \Nd\}$.  In
the present work, the inner optimization problem is solved via inexact
Gauss--Newton-CG with line search.  The cost of each Gauss--Newton iteration is
dominated by the CG solves for the search direction.  When using the prior
covariance operator as a preconditioner, the number of CG iterations is bounded
by $\mathcal{O}(\Nact)$; see e.g.,~\cite{IsaacPetraStadlerEtAl15}.  Each CG
step in turn requires two linearized PDE solves (incremental forward/adjoint
solves).  Hence, the cost of the $\Nd$ MAP estimation problems, in terms of
``forward-like'' PDE solves, is $\mathcal{O}(2 \times \Nd \times
n_\text{newton} \times \Nact)$, where $n_\text{newton}$ is the (average) number
of Gauss--Newton iterations.  The $2\times \Nd \times \Ntr$ solves in
\cref{equ:hess_sys_trace_est} are also done using preconditioned CG. As noted
before, each of these solves requires $\mathcal{O}(\Nact)$ CG iterations. To
summarize, the overall cost of evaluating
\cref{equ:design_criterion_mc_trace_est_def} is bounded by
\begin{equation}\label{equ:TR_cost}
\underbrace{\mathcal{O}(2 \times \Nd \times n_\text{newton} \times \Nact)}_{\text{cost of MAP point solves 
   in \cref{equ:map_estimation_trace_est}}}
+ 
\underbrace{\mathcal{O}(2\times\Nd \times \Ntr \times \Nact)}_{\text{cost of solves 
   in~\cref{equ:hess_sys_trace_est}}}. 
\end{equation}
It is important to note that MAP estimation problems as well as the linear solves in 
\cref{equ:hess_sys_trace_est} can be performed in parallel.

\paragraph{Cost of evaluating $\PhiNdLR$}
As is the case with
\cref{equ:design_criterion_mc_trace_est_def}, the solution of the MAP estimation 
problems~\cref{equ:map_estimation_LR} dominates the computational cost of evaluating $\PhiNdLR$.
This cost was analyzed for the case of $\PhiNdTR$. We next discuss the cost of solving the
eigenvalue problems in~\cref{equ:eig_LR}. We rely on the Lanczos method~\cite{GolubVanLoan13} for these eigenvalue 
problems.\footnote[5]{Another option is the use of randomized methods~\cite{HalkoMartinssonTropp11}.}
For each $i \in \{1, \ldots, \Nd\}$, the Lanczos method requires $\mathcal{O}(\Nact)$ 
applications of $\HT^i$ on vectors, each costing two PDE solves. Therefore, solving the
eigenvalue problems requires a total of $\mathcal{O}(2 \times \Nd \times \Nact)$ PDE solves.
Hence, the overall cost of computing $\PhiNdLR$ is bounded by
\begin{equation}\label{equ:LR_cost}
\underbrace{\mathcal{O}(2 \times \Nd \times n_\text{newton} \times \Nact)}_{\text{cost of MAP point solves
   in \cref{equ:map_estimation_LR}}}
+
\underbrace{\mathcal{O}(2\times\Nd \times \Nact)}_{\text{cost of solves
   in~\cref{equ:eig_LR}}}.
\end{equation}

To sum up, the computational cost of evaluating both $\PhiNdTR$ and $\PhiNdLR$ is dominated
by the cost of solving the MAP estimation problems.  However, by exploiting the low-rank structure of
the operators $\HT^i$, $\PhiNdLR$ is more efficient to compute, as it does not
require $\Ntr$ Hessian solves (compare also the second terms in 
~\cref{equ:TR_cost} and~\cref{equ:LR_cost}). Moreover, computing the traces using the
eigenvalues will be, in general, more accurate than a sampling based approach,
as long as sufficiently many eigenvalues are used.  Note also that when using a
greedy approach, $\Nact$ starts at $\Nact = 1$ in the first step and increase
by one in each iteration. 

However, the idea of using a randomized trace estimator does have some merits.  For
one thing, typically a small (in order of tens) $\Ntr$ is sufficient when
solving the OED problem.  Moreover, $\PhiNdTR$ is simpler to implement as it
does not require solving eigenvalue problems.  Note, however, that if a large
$\Ntr$ is needed, then the cost of Hessian solves in
\cref{equ:hess_sys_trace_est} might exceed the cost of the MAP estimation
problems.

\section{Model problem}
\label{sec:model}
Here, we present the model problem used to study our approach for OED under
uncertainty. We consider a nonlinear inverse problem governed by a linear
elliptic PDE, in a three-dimensional domain. This
problem, which is adapted from~\cite{NicholsonPetraKaipio18}, is motivated by
heat transfer applications. 
In \Cref{sec:inverse}, we detail the description of the Bayesian inverse
problem under study.  Subsequently, we detail the description of the OED objective
$\PhiNdLR$, for this specific model problem in~\Cref{sec:oed-objective}.


\begin{figure}[ht!]
\centering
\includegraphics[width=0.45\textwidth]{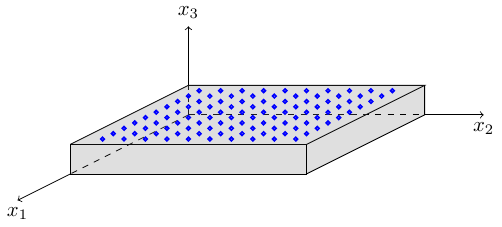}
\caption{Sketch of the physical domain $\Omega$ and location of candidate sensor
locations (blue circles).}
\label{fig:dom}
\end{figure}

\subsection{The Bayesian inverse problem}
\label{sec:inverse}

We consider the following model 
\begin{equation}\label{equ:poi}
  \begin{split}
    -\grad \cdot (\Exp{\iparb} \grad u)                   &= 0 \quad \text{ in } \Omega,\\
    u                                                     &= 0 \quad \text{ on } \GD, \\
    \Exp{\iparb} \grad{u} \cdot \vec{n}                   &= h \quad \text{ on } \GN, \\
    \Exp{\iparb} \grad{u} \cdot \vec{n} + \Exp{\iparm} u  &= 0 \quad \text{ on } \GR.
  \end{split}
\end{equation}
In this problem, 
$\Omega \subset \R^3$ is bounded domain with sufficiently smooth boundary
$\Gamma = \GD \cup \GN \cup \GR$, where $\GD$, $\GN$, and $\GR$ are mutually
disjoint.  In \cref{equ:poi}, $u$ is the state variable, $\iparm$ is the
inversion parameter, and $\iparb$ is the secondary uncertain parameter. The Neumann
boundary data is $h\in L^2(\GN)$, $\vec{n}$ is the unit-length outward normal
for the boundaries $\GN$ and $\GR$, and for simplicity we have considered zero
volume source, and homogeneous Dirichlet and Robin conditions.  In our
numerical experiments, $\Omega$ is a three-dimensional domain with a unit
square base and height such that aspect ratio (base/height) is 100;
see~\Cref{fig:dom}. In the present example, $\GR$ is the bottom edge of the domain, $\GN$
is the top edge, and $\GD$ is the union of the side edges.
In our computations, we define the boundary source term $h$ according to $h(\vec{x}) =
1 + \sin\big(4\pi \sqrt{(x_1 - 1)^2 + (x_2 - 1)^2}\big)$.

Note that in this problem, the inversion and secondary parameters are both
functions.  The inversion parameter belongs to the Hilbert space $\hilbm =
L^2(\GR)$, which is equipped with the $L^2(\GR)$-inner product.  The secondary
parameter $\iparb$ takes values in $L^2(\Omega)$. As discussed below, we define
$\xi$ as an $L^2(\Omega)$-valued Gaussian random variable with almost surely
continuous realizations, which  ensures (almost sure) well-posedness of the
problem~\cref{equ:poi}.  Below we use $\ipp{\Omega}{\cdot}{\cdot}$ and
$\ipp{\GN}{\cdot}{\cdot}$ to denote the $L^2(\Omega)$- and $L^2(\GN)$-inner
products, respectively.  Also, with a slight abuse of notation, we denote the
$L^2(\Omega)^3$ inner-product with $\ipp{\Omega}{\cdot}{\cdot}$ as well.



We use measurements of $u$ on the top boundary, $\GN$, to estimate the
inversion parameter $\iparm$. The measurements are collected at a set of sensor
locations as depicted in~\Cref{fig:dom}.  Hence, in the present problem, the
parameter-to-observable is defined as the composition of a linear observation
operator $\B$, which extracts solution values at the sensor locations, and the
PDE solution operator. Note that this parameter-to-observable map is a
nonlinear function of the inversion and secondary
parameters.

As detailed in~\Cref{sec:BAE}, we assume a Gaussian noise model,
and use the BAE approach to account for the secondary uncertainties in the
inverse problem.  Also, we use a Gaussian prior law for $\iparm$
and model the secondary uncertainty as a Gaussian random variable.\footnote[7]{As
discussed earlier, the presented framework does not rely on a specific choice
of distribution law for $\iparb$. The Gaussian assumption here is merely for
computational convenience.} For both $\iparm$ and $\iparb$, we use covariance
operators that are defined as negative powers of Laplacian-like
operators~\cite{Stuart10}; see \Cref{sec:numerics}.  
Note that with the appropriate choice of the covariance operator,
the realizations of $\iparb$ will be almost surely continuous on the closure of
$\Omega$.


\subsection{OED under uncertainty problem formulation}\label{sec:oed-objective}
In this section we discuss the precise definition of $\PhiNdLR$, for the model
inverse problem discussed in \Cref{sec:inverse}.  The MAP estimation problems
in~\cref{equ:map_estimation_LR} require minimizing
$\mathcal{J}_{\vec{w}}(\iparm;\obs^i)$ defined in~\cref{equ:costw}. 
For notational convenience, in this section, we use the shorthand $\mi$ to denote 
$\iparmap^\yhati$.
In the
present example, the first order optimality conditions for this optimization
problem are given by
\begin{align}
  \ipp{\Omega}{\Exp{\iparb} \grad{u_i}}{\grad{\utest}} - \ip{h}{\utest}_{\GN} +
  \mip{\Exp{\mi} u_i}{\utest} &= 0, &\quad \forall \utest \in \V,
  \label{equ:state} \\
  \ipp{\Omega}{\Exp{\iparb} \grad{p_i}}{\grad{\ptest}} + \mip{\Exp{\mi}
    \ptest}{p_i} + \ipp{\Omega}{\B^* \WW(\vec w) (\B u_i - \obs^i +
    \vec{\eps}_0)}{\ptest} &= 0, &\quad \forall \ptest \in \V,
  \label{equ:adj}\\
  \cip{\mi - \iparprm}{\mtest} + \mip{\mtest \Exp{\mi}
    u_i}{p_i} &= 0, &\quad \forall \mtest \in \CM,
  \label{equ:grad}
\end{align}
where $\V =  \{ v \in H^1(\Omega) : \restr{v}{\GD} = 0\}$.
For the derivation details,
we refer to~\cite{NicholsonPetraKaipio18}.  Note
that~\cref{equ:state} is the weak form of the state
equation~\cref{equ:poi}, and \cref{equ:adj} and \cref{equ:grad} are the 
weak forms of the 
adjoint and gradient equations, respectively.
Note also that the left hand side of \cref{equ:grad}
describes the action of the derivative of
$\mathcal{J}_{\vec{w}}(\iparm;\obs^i)$ in the direction $\mtest$.

To specify the eigenvalue problem in~\cref{equ:eig_LR}, we first
discuss the action of the operator $\H^i$ (evaluated at $\mi$) in
the direction $\mhat$. In weak form, this Hessian application satisfies, 
\begin{align}\label{equ:GNhessmisfitapply}
   \mip{\H^i(\mi)\mhat}{\mtest} = \mip{\mtest \Exp{\mi} u_i}{\hat{p}_i}, 
\end{align}  
for every $\mtest \in \CM$.
In~\cref{equ:GNhessmisfitapply}, 
for a given $\mi$, $ u_i$ solves the state
problem~\eqref{equ:poi}, $\hat{p}_i$ solves the 
incremental adjoint problem
\begin{align}
  \ipp{\Omega}{\Exp{\iparb} \grad{\hat{p}_{i}}}{\grad{\ptest}} +
  \mip{\Exp{\mi} \ptest}{\hat{p}_{i}} + \ipp{\Omega}{\B^* \WW(\vec w)
    \B \hat{u}_{i}}{\ptest} = 0, \quad \forall \ptest \in \V,
\end{align}
and $\hat{u}_{i}$ solves the so-called incremental state problem
\begin{align}
  \ipp{\Omega}{\Exp{\iparb} \grad{\hat{u}_{i}}}{\grad{\utest}} +
  \mip{\Exp{\mi} \hat{u}_{i}}{\utest} + \ip{\Exp{\mi}
    \mhat u_i}{\utest}_{\GR}= 0, \quad \forall \utest \in \V.
\end{align}

We next summarize the OED problem of minimizing~\eqref{equ:design_criterion_LR_def}
as a PDE-constrained optimization problem specialized for the model inverse problem
in \Cref{sec:model}:
\begin{subequations}
  \label{equ:oeduu-ex}
  \begin{align}
    &\min_{\vec{w} \in \{0, 1\}^\Nd} -\frac{1}{\Nd}\sum_{i=1}^\Nd
    \sum_{k=1}^r \frac{\lambda_{ik}}{1+\lambda_{ik}} \| s_{ik}\|_\hilbm^2,
  \end{align}
  where, for $i \in \{1, \ldots, \Nd\}$ and $k \in \{1, \ldots, r\}$,
    \begin{align}
      \ipp{\Omega}{\Exp{\iparb} \grad{u_i}}{\grad{\utest}} - \ipp{\GN}{h}{\utest} +
      \mip{\Exp{\mi} u_i}{\utest} &= 0, &\forall \utest \in \V,
      \label{equ:oeduu-ex-state}\\
      \ipp{\Omega}{\Exp{\iparb} \grad{p_i}}{\grad{\ptest}} + \mip{\Exp{\mi}
        \ptest}{p_i} + \ipp{\Omega}{\B^* \WW(\vec w) (\B u_i - \obs^i +
        \vec{\eps}_0)}{\ptest} &= 0, &\forall \ptest \in \V,
      \label{equ:oeduu-ex-adj}\\
      \cip{\mi - \iparprm}{\mtest} + \mip{\mtest \Exp{\mi} u_i}{p_i} &= 0, 
      &\forall \mtest \in \CM, 
      \label{equ:oeduu-ex-grad}\\
      \ipp{\Omega}{\Exp{\iparb} \grad{\hat{p}_{ik}}}{\grad{\ptest}} + 
        \mip{\Exp{\mi} \ptest}{\hat{p}_{ik}} + 
        \ipp{\Omega}{\B^* \WW(\vec w) \B \hat{u}_{ik}}{\ptest} &= 0, &\forall \ptest \in \V,
      \label{equ:oeduu-ex-incadj}\\
      \mip{\mtest \Exp{\mi} u_i}{\hat{p}_{ik}} &= 
         \lambda_{ik} \cip{s_{ik}}{\mtest}, &\forall \mtest \in \CM, 
      \label{equ:oeduu-ex-hessapp}\\
      \ipp{\Omega}{\Exp{\iparb} \grad{\hat{u}_{ik}}}{\grad{\utest}} + 
         \mip{\Exp{\mi} \hat{u}_{ik}}{\utest} + \mip{\Exp{\mi} s_{ik} u_i}{\utest} &= 0, 
         &\forall \utest \in \V,
      \label{equ:oeduu-ex-incstate}\\
      \cip{s_{ik}}{s_{ik}} &= 1.
    \end{align}
\end{subequations}
The PDE
constraints~\eqref{equ:oeduu-ex-state}--\eqref{equ:oeduu-ex-grad} are
the optimality system~\eqref{equ:state}--\eqref{equ:grad}
characterizing the MAP point $\mi = \iparmap^\yhati$ described
in~\eqref{equ:map_estimation_LR}. The equations
\eqref{equ:oeduu-ex-incadj}--\eqref{equ:oeduu-ex-incstate} are the PDE
constraints that describe the Hessian apply and eigenvalue problem.
Note that we have reformulated the eigenvalue problem according to~\cref{equ:eig_gen}.

\section{Computational results}\label{sec:numerics}
In this section, we numerically study our OED under uncertainty approach, which
we apply to the model inverse problem described in~\Cref{sec:model}.  We begin
by specifying the Bayesian problem setup and discretization details
in~\Cref{sec:setup}. Then, we study the impact of secondary model uncertainty
on the measurements and compute the BAE error statistics
in~\Cref{sec:BAE_numerics}.  Finally, in~\Cref{sec:OEDUU}, we examine the
effectiveness of our approach in computing uncertainty aware designs.  We also
study the impact of ignoring model uncertainty in (i) experimental design stage
and (ii) both experimental design and inference stages.  
Ignoring uncertainty in the design stage amounts to fixing the secondary
parameter $\iparb$ at its nominal value (i.e., using the approximate
model~\cref{equ:model_approx}) and ignoring the approximation error when
solving the OED problem. We refer to designs computed in this manner as
\emph{uncertainty unaware} designs. Note that in the case of (i), the uncertainty
is still accounted for when solving the inverse problem. This study
illustrates the importance of accounting for model uncertainty, when computing
experimental designs.  On the other hand, the study of case (ii) 
illustrates the
pitfalls of ignoring uncertainty in both the optimal
design problem and subsequent solution of the inverse problem.

\subsection{Problem setup}\label{sec:setup}
We consider $\Ns = 100$ candidate sensor locations that are arranged in a
regular grid on the top boundary $\GN$; see~\Cref{fig:dom}.  The additive noise
in the synthetic measurements has a covariance matrix of the form
$\ncov=\sigma^2\mat{I}$, with $\sigma = 10^{-3}$. This amounts to about one
percent noise.  

We use a Gaussian prior law $\prior =
\GM{\iparprm}{\Cprior}$ for $\iparm$. The prior mean is taken as a constant
function $\iparprm \equiv 1$, and we use a covariance operator given by the
inverse of a Laplacian-like operator.  Specifically, we let
$\Cprior:=\mathcal{A}^{-2}$, with
$\mathcal{A}[\iparm]=-\nabla\cdot(\theta\nabla\iparm)+\alpha\iparm$, where we
take $\theta=0.1$ and $\alpha=1$. To help mitigate undesirable boundary affects
that can arise due to the use of PDE-based prior covariance operators, we equip
the operator $\mathcal{A}$ with Robin boundary conditions~\cite{DaonStadler18}.
For illustration, four random draws from the prior distribution are shown in
\Cref{fig:prior_m}.
\begin{figure}[ht!]
\centering
\includegraphics[height=.22\textwidth]{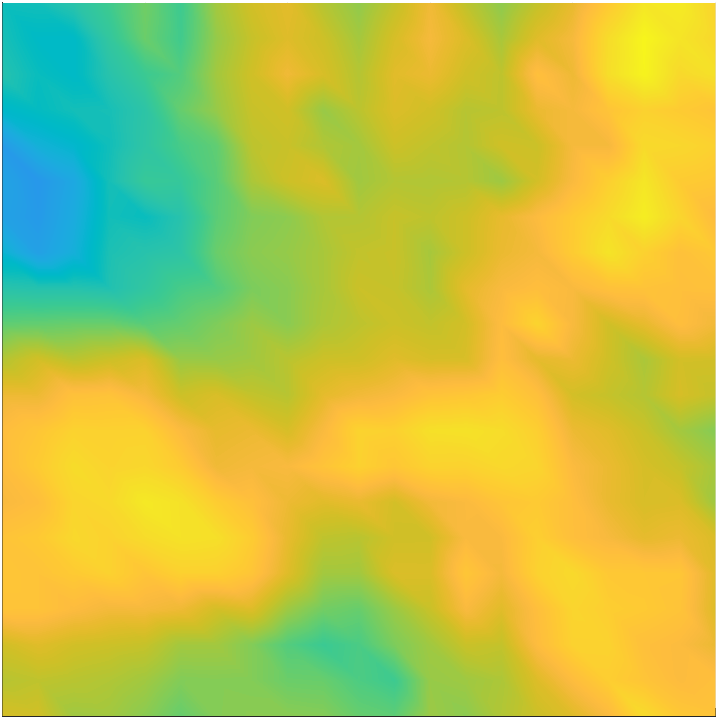}
\hfill
\includegraphics[height=.22\textwidth]{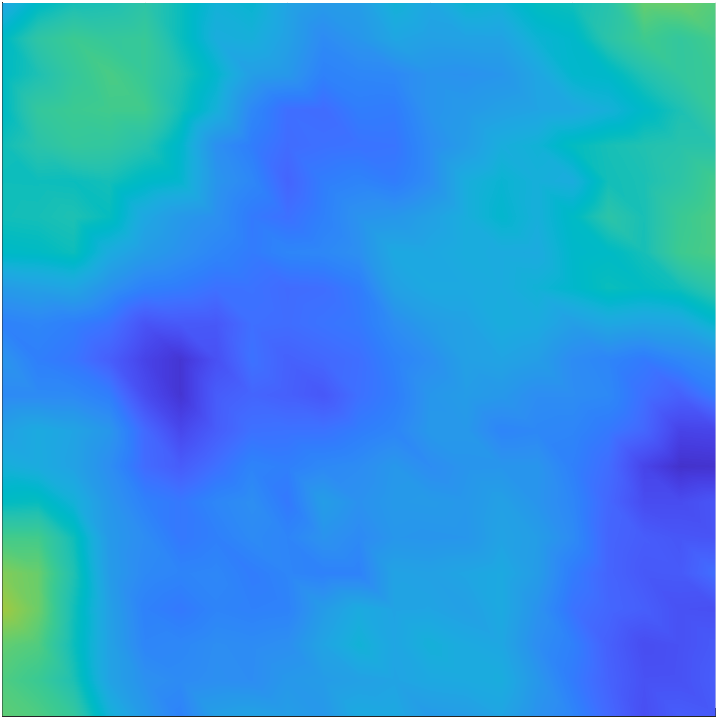}
\hfill
\includegraphics[height=.22\textwidth]{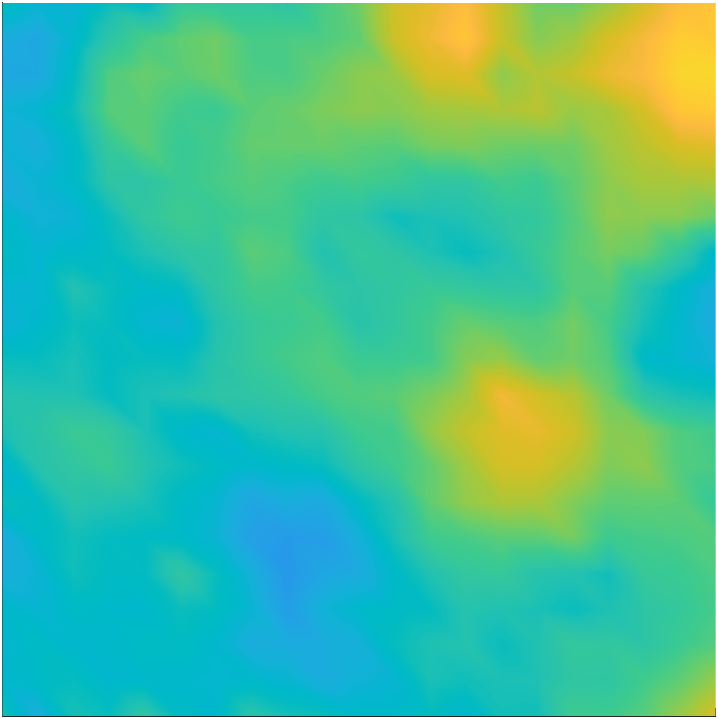}
\hfill
\includegraphics[height=.22\textwidth]{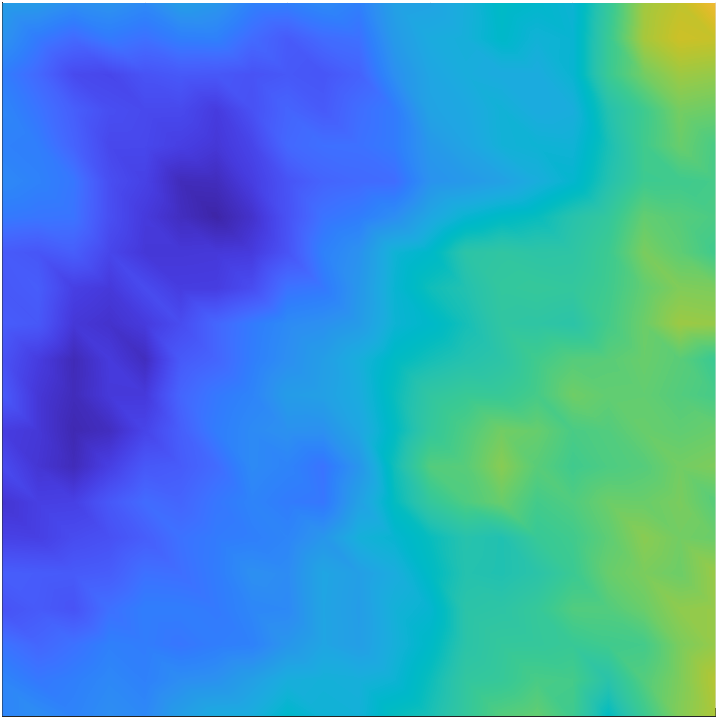}
\hfill
\includegraphics[height=.22\textwidth]{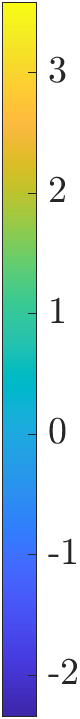}
\caption{Samples of the primary uncertain parameter $m$.}
\label{fig:prior_m}
\end{figure}
The law of the secondary parameter $\iparb$ is chosen to be a Gaussian
measure $\mu_\iparb = \GM{\bar\iparb}{\Caux}$.  
We set the mean of $\iparb$ as the constant function
$\bar\iparb\equiv0$, and the covariance operator is defined as
$\Caux:=\mathcal{L}^{-2}$, where
$\mathcal{L}[\iparb]=-\nabla\cdot(\mat{\Theta}\nabla\iparb)+\gamma\iparb$
with $\mat{\Theta}=0.25\diag(1,1,\tfrac{1}{100})$ and $\gamma=50$.
This choice of $\mat{\Theta}$ corresponds to a random field with much
shorter correlation in the $z-$direction than in the $x$- and
$y$-directions, inline with aspect ratio of 100 used in defining the
domain $\Omega$. We show four representative
samples of $\iparb$ in \Cref{fig:prior_xi}.

\begin{figure}[ht!]
\centering
\includegraphics[width=.24\textwidth]{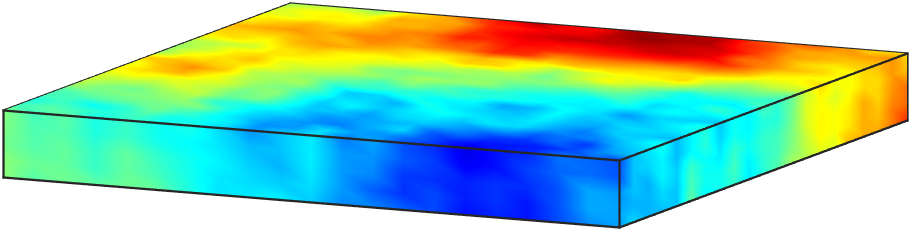}
\hfill
\includegraphics[width=.24\textwidth]{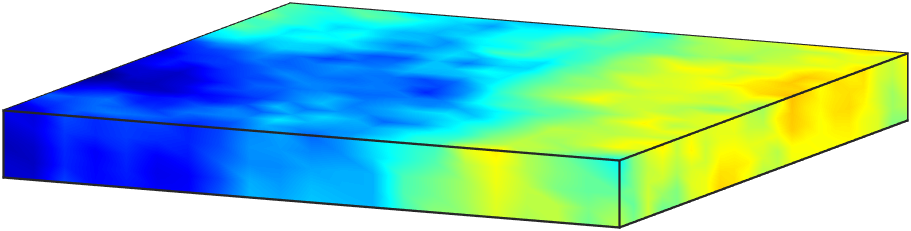}
\hfill
\includegraphics[width=.24\textwidth]{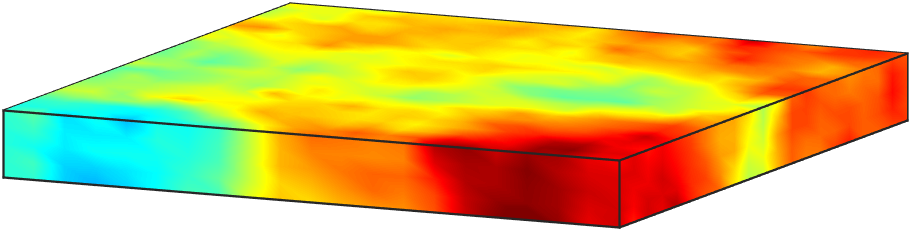}
\hfill
\includegraphics[width=.24\textwidth]{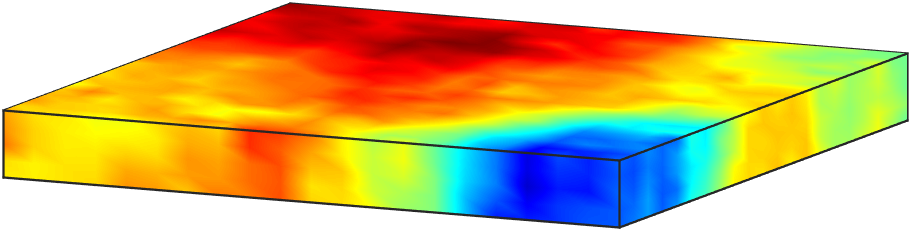}
\\
\vspace{5mm}
\includegraphics[width=.45\textwidth]{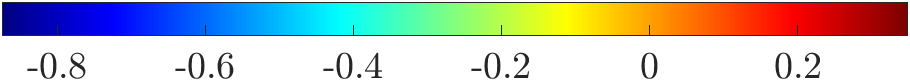}
\caption{Samples of the secondary uncertain parameter $\iparb$.} 
\label{fig:prior_xi}
\end{figure}

The forward problem \cref{equ:poi} is solved using a continuous Galerkin finite
element method.  We use $9{,}600$ tetrahedral elements and $2{,}205$ piecewise
linear basis functions. As such, the discretized state and adjoint variables,
as well as the secondary parameter $\iparb$, have dimension $2{,}205$. On
the other hand, the discretized inversion parameter $m$, which is defined on
the bottom boundary, is of dimension $441$. 

\subsection{Incorporating the model uncertainty in the inverse problem}\label{sec:BAE_numerics}
We begin by studying the impact of the secondary parameter on the
solution of the forward problem. This is illustrated 
in~\Cref{fig:xi_impact}.  In this experiment, we solve the forward problem for
a fixed $\iparm$ and four different realizations of $\iparb$. Note
that, although the qualitative behavior of $u$ on $\GN$ is similar for the
different samples of $\iparb$, there are considerable differences in the values.
This indicates that the approximation error due to fixing $\xi$ at the nominal
value will have significant variations.
\begin{figure}[t!]
\centering
\includegraphics[height=.22\textwidth]{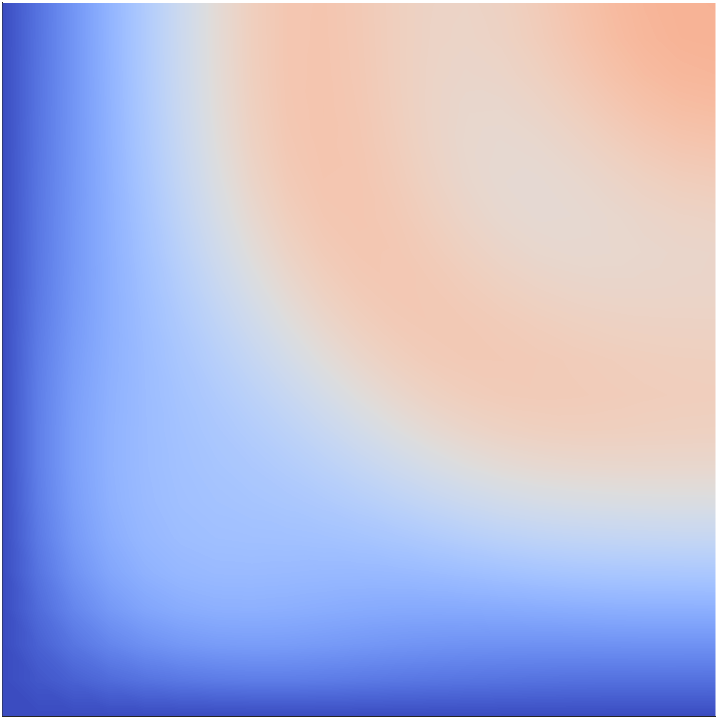}
\hfill
\includegraphics[height=.22\textwidth]{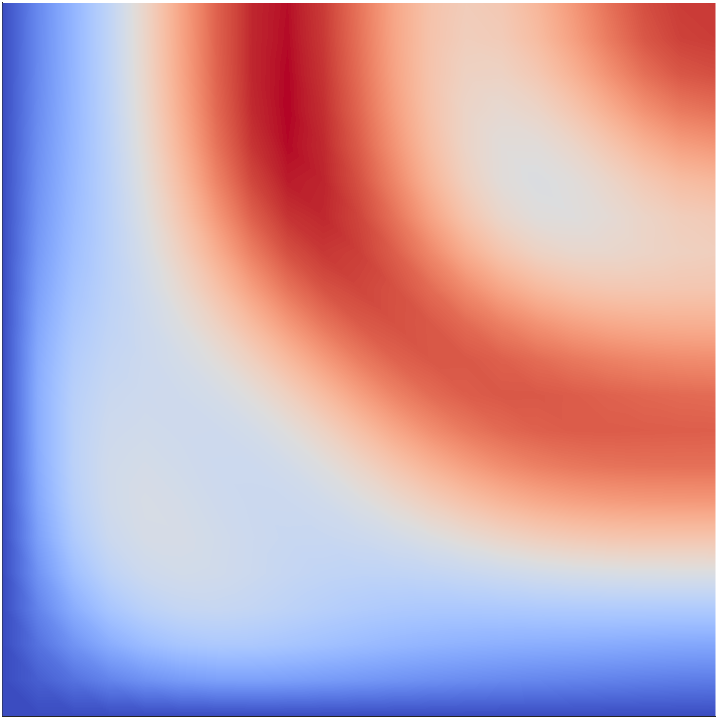}
\hfill
\includegraphics[height=.22\textwidth]{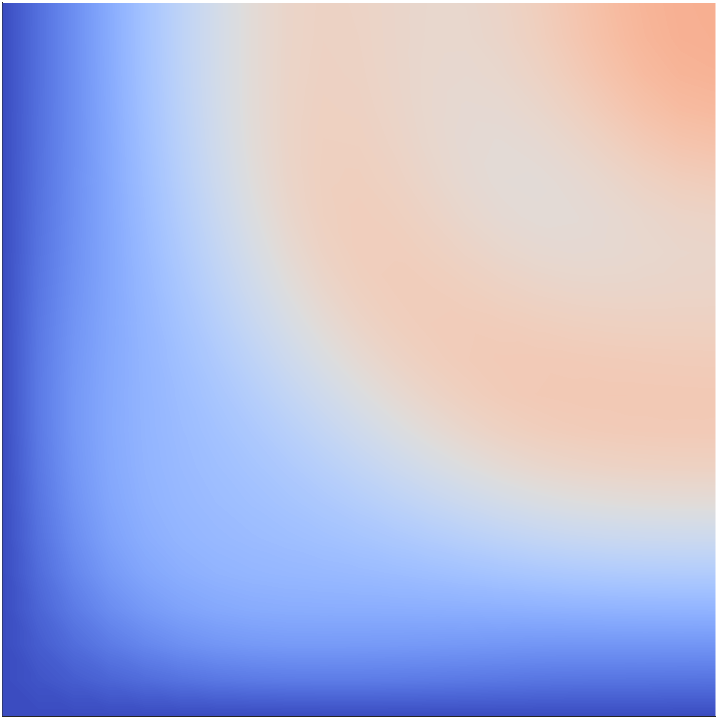}
\hfill
\includegraphics[height=.22\textwidth]{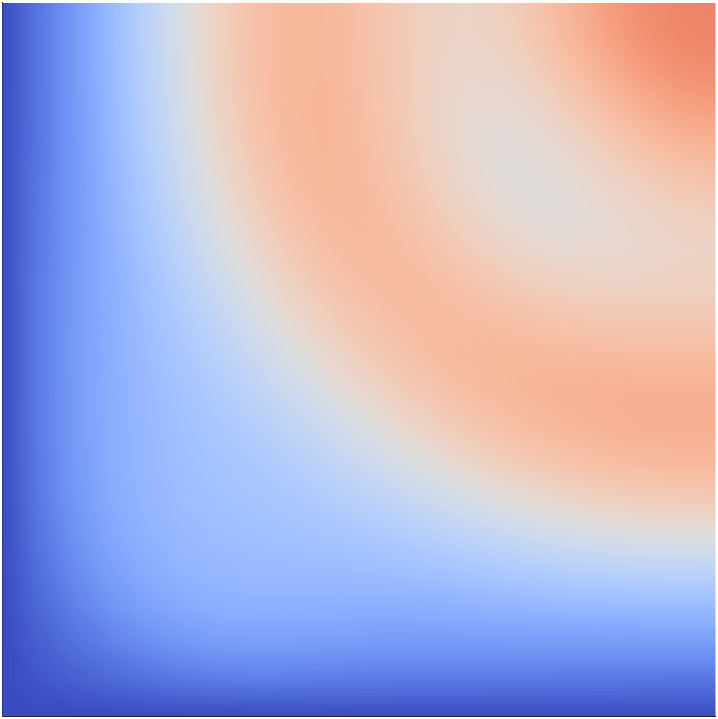}
\hfill
\includegraphics[height=.22\textwidth]{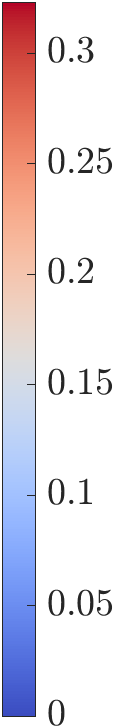}
\caption{The effect of the secondary uncertain parameter $\xi$ on the
  state $u$ on top surface for the (fixed) true primary parameter $m$
  and four different realizations of $\xi$.}
\label{fig:xi_impact}
\end{figure}

We compute the sample mean and covariance matrix of the approximation error as
in~\cref{equ:BAE_comp} with $\Nsamp = 1{,}000$.  The mean and marginal standard
deviations are shown in \Cref{fig:approxerrs}. To illustrate the correlation
structure of the approximation error, we show the correlation of
the approximation error at two of the candidate locations with 
the other candidate locations in
\Cref{fig:corr}~(left-middle).  To provide an overall picture, we report the
correlation matrix of the approximation error
in~\Cref{fig:corr}~(right). Note that the approximation errors at the
sensor sites are not only larger than the measurement noise, but also they are
highly correlated and have a nonzero and non-constant mean.  Thus, in
the present application, the approximation error due to model uncertainty
cannot be ignored and needs to be accounted for.

\begin{figure}[ht!]
\centering
\includegraphics[height=.25\textwidth]{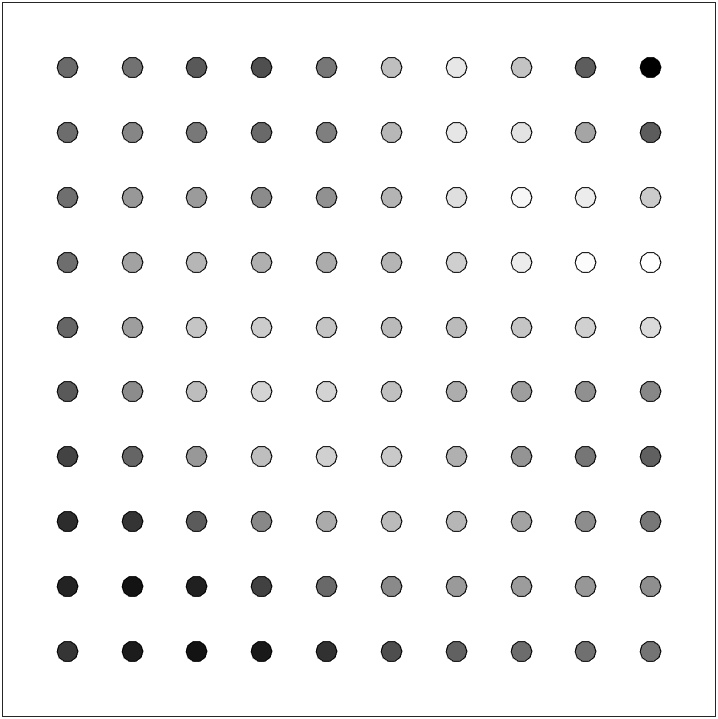}
\includegraphics[height=.25\textwidth]{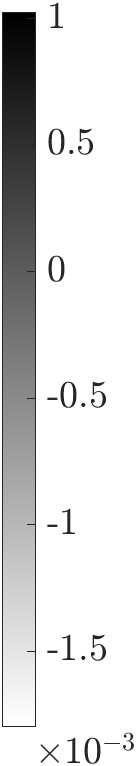}
\hspace{5mm}
\includegraphics[height=.25\textwidth]{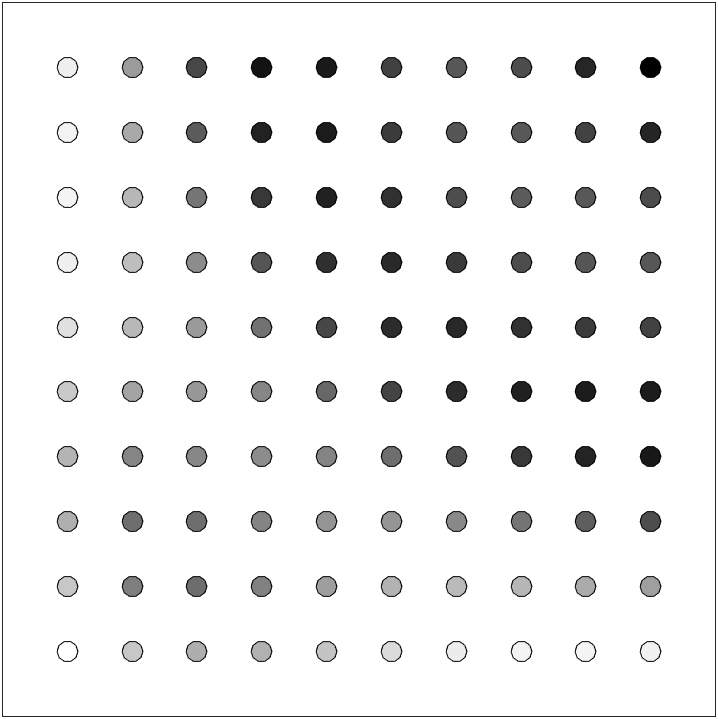}
\includegraphics[height=.25\textwidth]{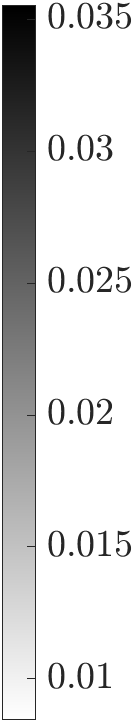}
\caption{ Left: the mean of the approximation error at the candidate
  sensor locations; right: the marginal standard deviation of the
  approximation error at these locations (i.e., the square root of
  diagonal entries of $\ecovhat$ given in~\cref{equ:BAE_comp}).}
\label{fig:approxerrs}
\end{figure}

\def \pos {0.5\columnwidth}
\addtolength\abovecaptionskip{-8pt}
\begin{figure}[ht]\centering
  \begin{tikzpicture}
    \node (1) at (0*\pos-0.53*\pos,    0*\pos){\includegraphics[height=.28\textwidth,trim=120 40 120 70,clip]{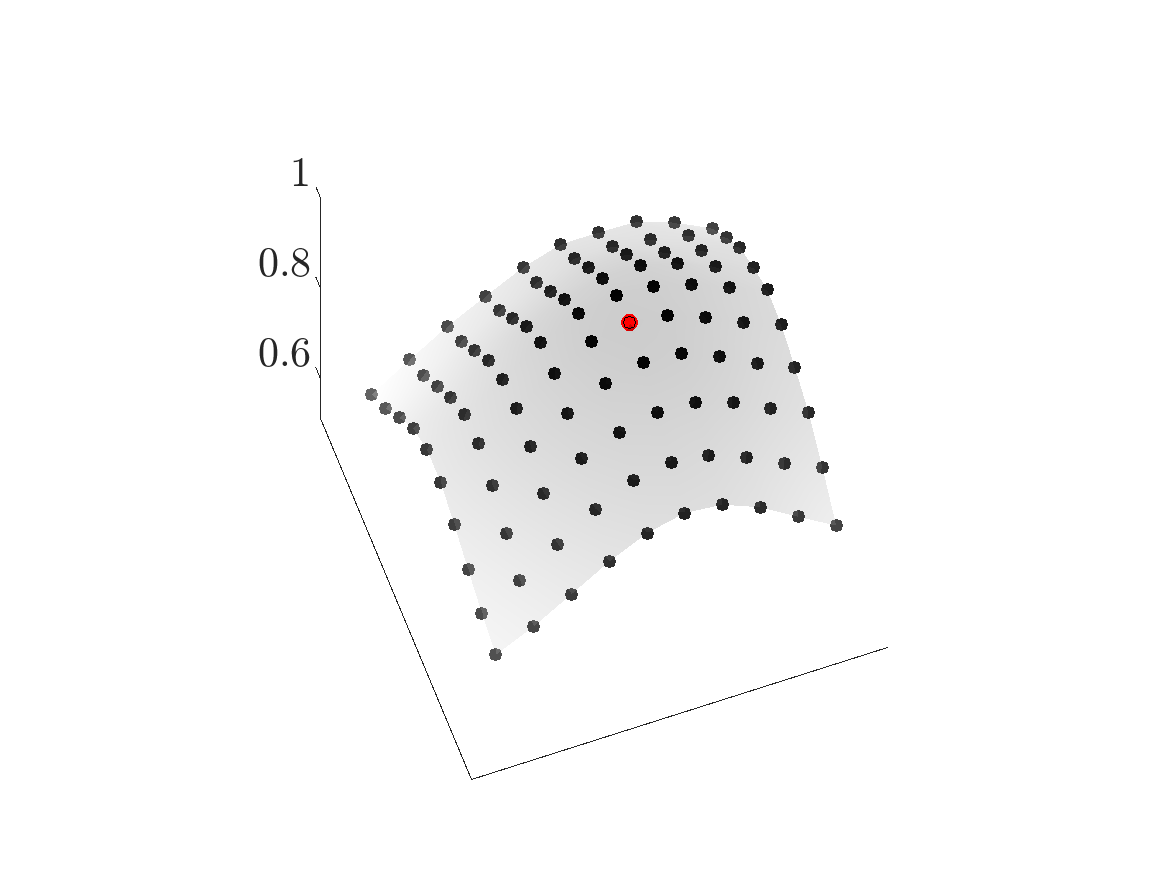}};
    \node (2) at (0*\pos, 0*\pos){\includegraphics[height=.28\textwidth,trim=120 40 120 70,clip]{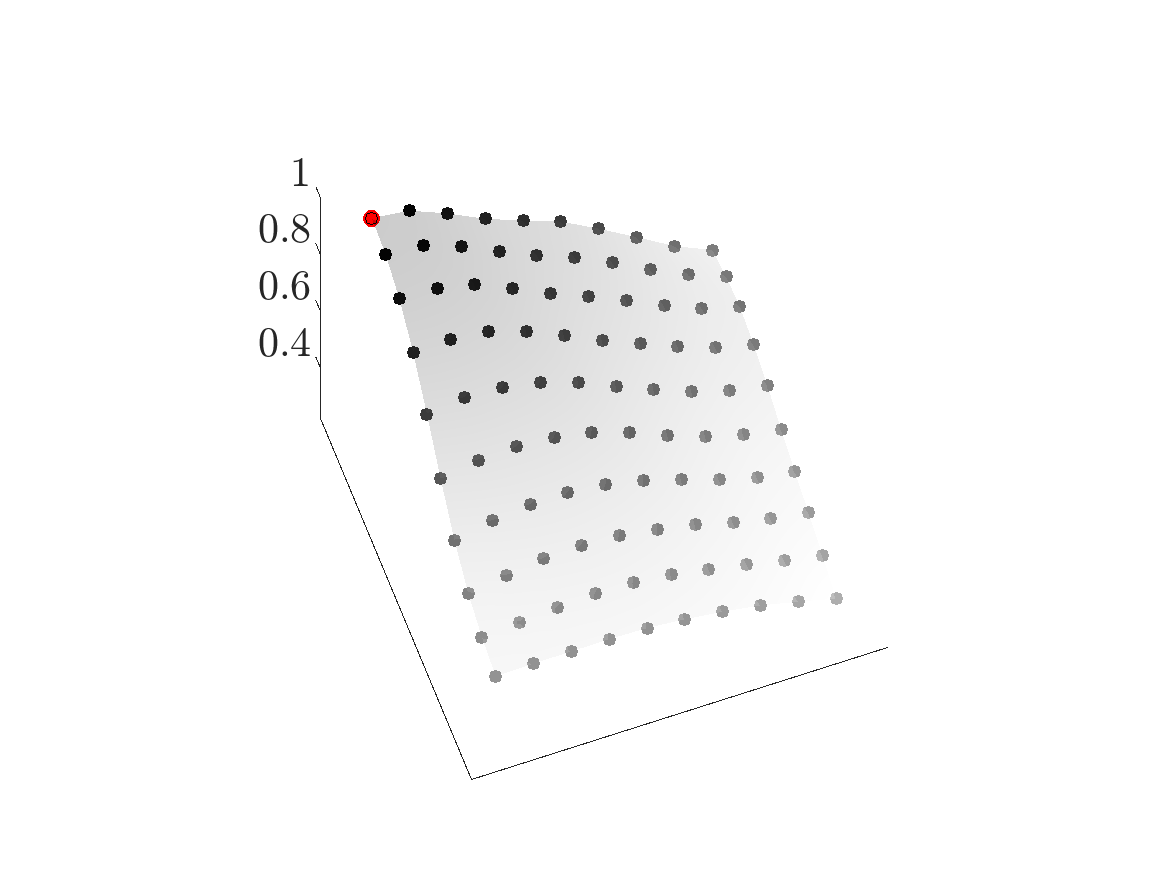}};
    \node (3) at (0*\pos+0.6*\pos, 0*\pos){\includegraphics[height=.28\textwidth]{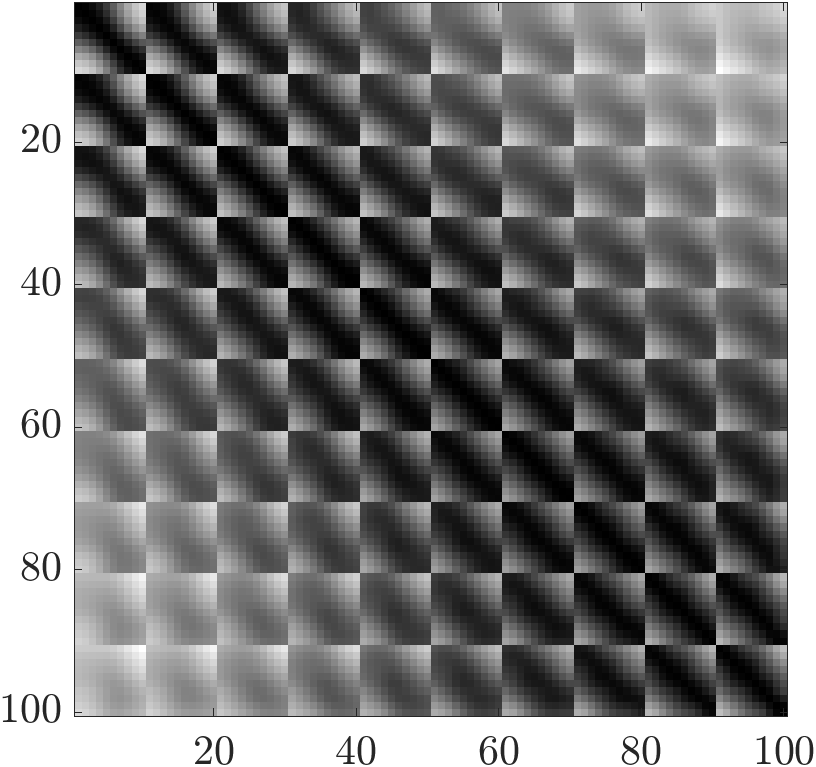}};
    \node (4) at (0*\pos+0.95*\pos,  0.0241*\pos){\includegraphics[height=.268\textwidth]{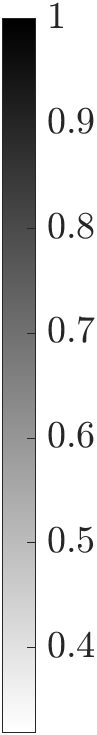}};
  \end{tikzpicture}
\caption{Left: The correlation of a measurement from a sensor near the
  middle marked by red dot, with the remaining sensors; middle: The
  same information for the sensor in the top left corner of the sensor
  grid; right: the correlation matrix of the approximation error.}
\label{fig:corr}
\end{figure}
\addtolength\abovecaptionskip{8pt}


\subsection{Optimal experimental design under uncertainty}\label{sec:OEDUU}
We begin by solving the OED problem \cref{equ:oeduu-ex} with $\Nd=5$ training data samples.
In~\Cref{fig:OED}~(left), we show an uncertainty aware optimal sensor placement with 20
sensors.  Note that due to the use of a greedy algorithm, we can track the
order in which the sensors are picked. To understand the impact of 
ignoring model uncertainty in the design stage, we also compute an uncertainty unaware
design; this is reported in~\Cref{fig:OED}~(right). 
\begin{figure}[ht!]
\centering
\includegraphics[height=.3\textwidth]{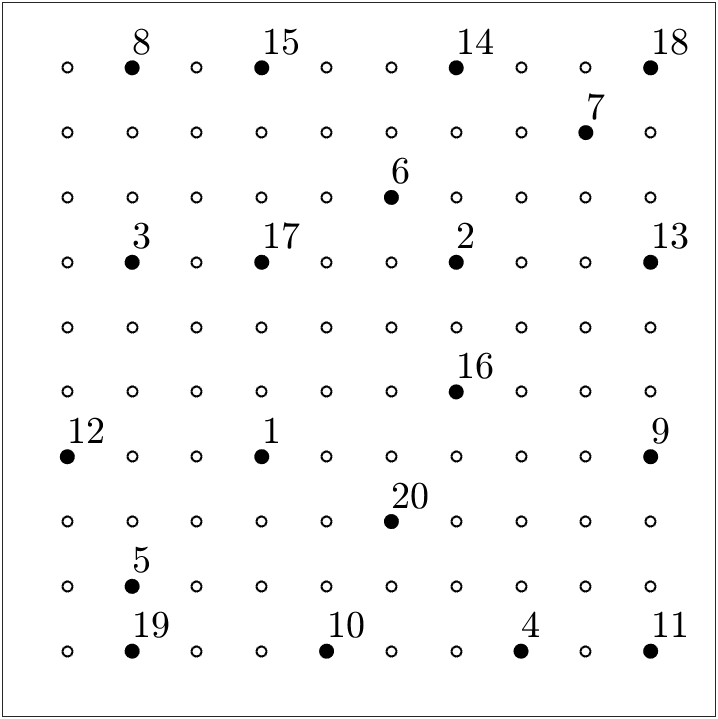}
\hspace{5mm}
\includegraphics[height=.3\textwidth]{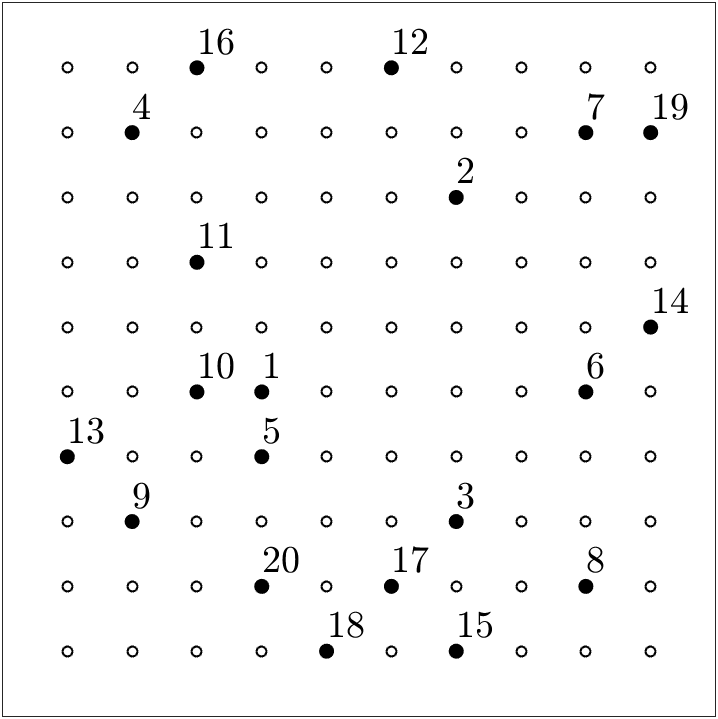}
\caption{The uncertainty aware (left) and uncertainty unaware (right) 
sensor placements. The numbers indicate the order in which the sensors
are picked in the greedy algorithms for minimizing~\cref{equ:design_criterion_LR}.}
\label{fig:OED}
\end{figure}

To evaluate the quality of the computed designs, we first study the expected
posterior variance and expected relative error of the MAP point. 
To facilitate this, we draw parameter samples 
$\{\iparm^{\text{v}}_i\}_{i=1}^\Nv$ from $\prior$ and generate validation data
samples, 
\[
\yhatip = \PPsi(\afwd(\iparm^{\text{v}}_i,\iparb^{\text{v}}_i) + \vec\eta^{\text{v}}_i),\quad
i \in \{1, \ldots, \Nv\},
\]
where $\{\iparb^{\text{v}}_i\}_{i=1}^\Nv$ and
$\{\vec\eta^{\text{v}}_i\}_{i=1}^\Nv$ are draws from $\mu_\iparb$ and
$\GM{\vec{0}}{\ncov}$, respectively. 
Then, we consider 
\begin{equation}\label{equ:diagnostics}
\bar{V}(\vec w) = \frac{1}{\Nv}\sum_{i=1}^{\Nv}\trace(\Cpost^\yhatip)
\quad \text{and} \quad
\bar{E}_{\scriptscriptstyle\text{MAP}}(\vec w) = 
\frac{1}{\Nv}\sum_{i=1}^{\Nv}
\frac{\|\iparmap^\yhatip-\iparm^{\text{v}}_i\|_\hilbm}{\|\iparm^{\text{v}}_i\|_\hilbm}.
\end{equation}
In our numerical experiments we use $\Nv=100$.
We compare $\bar{V}(\vec w)$ and $\bar{E}_{\scriptscriptstyle\text{MAP}}(\vec
w)$ when solving the inverse problem with the computed optimal design versus
randomly chosen designs in~\Cref{fig:cloud}, where we consider designs with
different numbers of sensors. We also examine the impact of ignoring the model
uncertainty in solving the OED problem (see the black dots in~\Cref{fig:cloud}).
Note that the uncertainty aware designs outperform the random designs as well as
the uncertainty unaware designs. This is most pronounced when the number of
sensors is small. This is precisely when optimal placement of sensors is
crucial.  We also observe that as the number of sensors in the designs increase,
the cloud moves left and downward. This is expected. The more sensors we use,
the more we can improve the quality of the MAP point and reduce posterior
uncertainty.  For $\Ns \geq 30$, we note that the difference between uncertainty
aware and uncertainty unaware designs become small (in the case of $\Ns = 30$, 
they nearly overlap). For $\Nd \geq 40$, even the random designs become
competitive. All these are expected.  As mentioned, earlier, optimal placement
of sensors is crucial, when we have access only to a ``small'' number measurement points.
What constitutes ``small'' is problem dependent.

\addtolength\abovecaptionskip{-5pt}
\begin{figure}[ht!]
\centering
\includegraphics[width=1\textwidth]{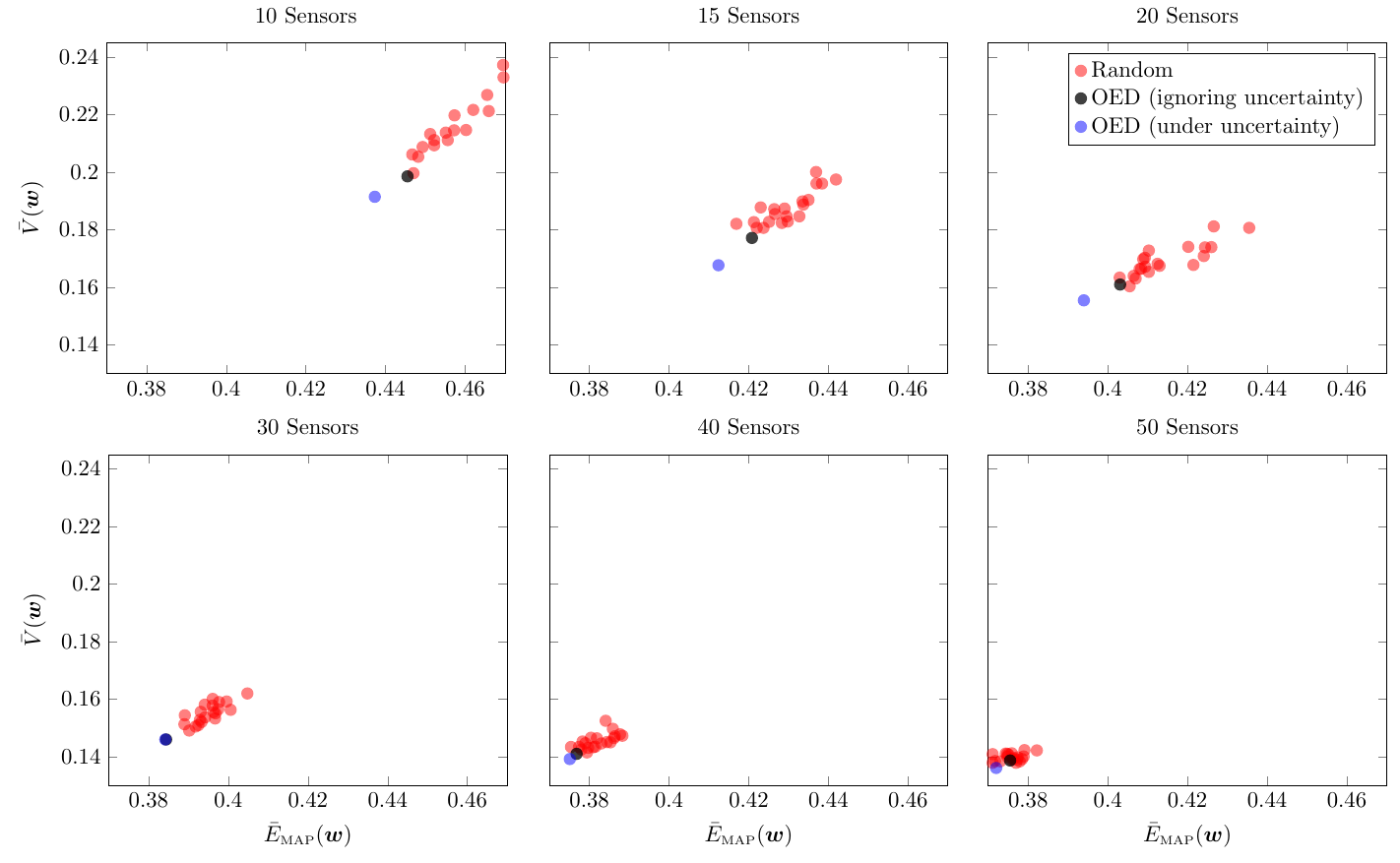}
\vspace{-5mm}
\caption{The expected relative error $\bar{E}_{\scriptscriptstyle\text{MAP}}(w)$ of the 
MAP point versus expected posterior variance $\bar{V}(w)$ 
using random designs (red dots), an uncertainty aware sensor placements (blue dots), and 
uncertainty unaware sensor placements (black dots) using 10, 15, and 20 sensors.} 
\label{fig:cloud}
\end{figure}
\addtolength\abovecaptionskip{5pt}

Next, we illustrate the effectiveness of the computed optimal designs in
reducing posterior uncertainty. To this end, we consider the computed optimal
designs with 10 sensors. In~\Cref{fig:uncertainty_reduction}, we show the
effectiveness of the uncertainty aware optimal design in reducing posterior
uncertainty; we also report the posterior standard deviation field, when
solving the inverse problem with an uncertainty unaware design.  To complement
this study, we consider the quality of the MAP points computed using uncertainty
aware and uncertainty unaware designs in~\Cref{fig:MAP_points}.  Overall, we
observe that the uncertainty aware design is more effective in reducing
posterior uncertainty and results in a higher quality MAP point. This
conclusion is also supported by the results reported in~\Cref{fig:cloud}. 

Note that the data used to solve the inverse problem is synthesized by solving
the PDE model~\cref{equ:model}, using our choice of the ``truth'' inversion
parameter (see~\Cref{fig:MAP_points}~(left)) and a randomly chosen $\iparb$
followed by extracting measurements at the sensor sites and adding measurement
noise.  This simulates the practical situation when field data that corresponds
to an unknown choice of $\iparb$ is collected.

\begin{figure}[ht!]
\centering
\includegraphics[height=.3\textwidth]{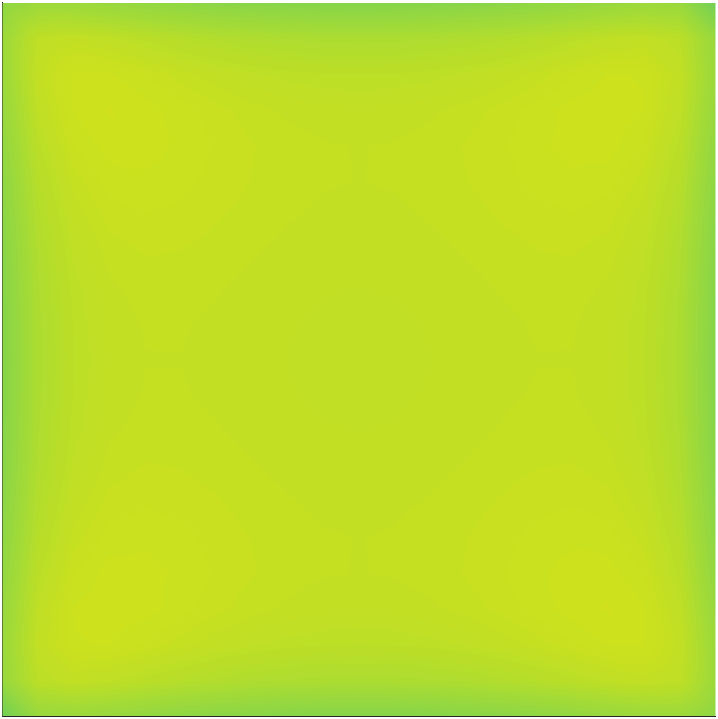}
\hfill
\includegraphics[height=.3\textwidth]{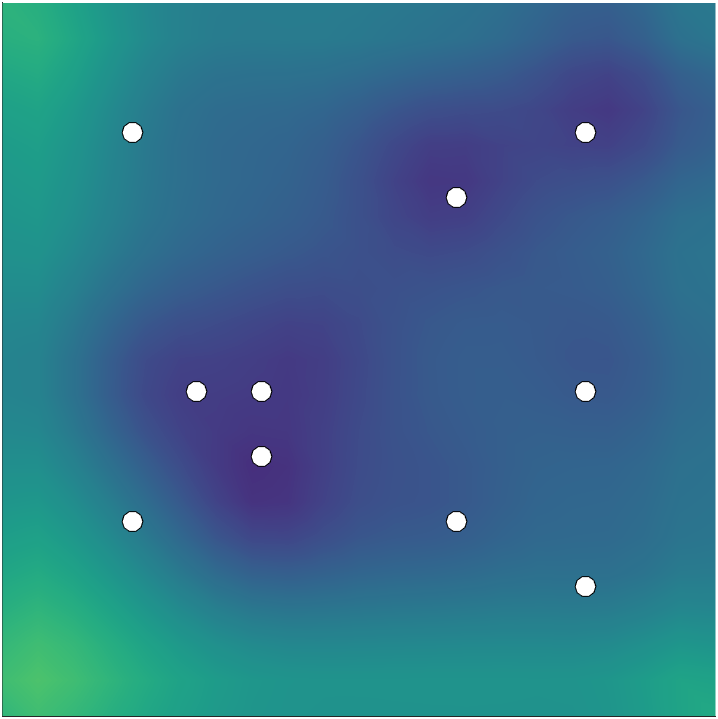}
\hfill
\includegraphics[height=.3\textwidth]{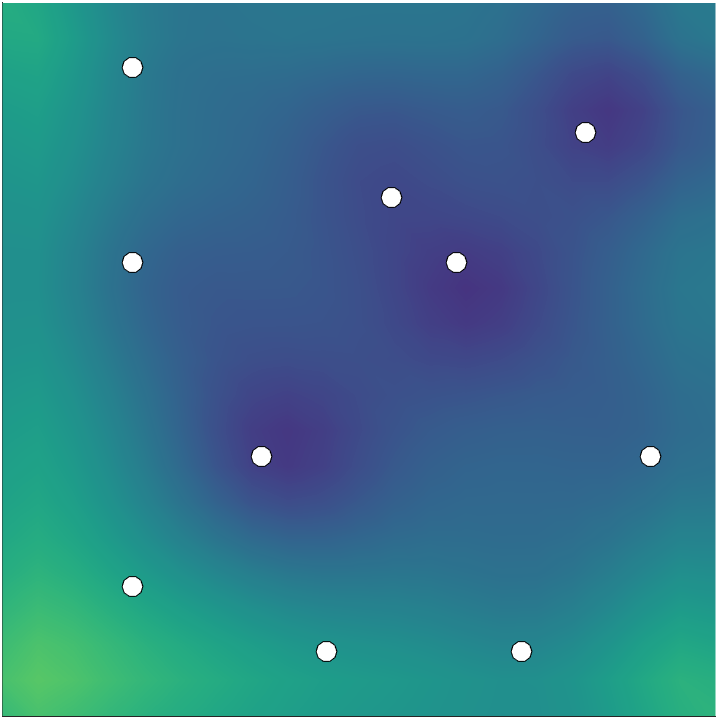}
\hfill
\includegraphics[height=.3\textwidth]{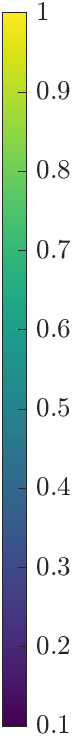}
\caption{The effectiveness of the computed optimal design with 10 sensors on posterior uncertainty. 
We report 
the pointwise prior standard deviation of
$m$ (left), and the pointwise posterior standard deviation of $m$ using the uncertainty aware 
(middle) and uncertainty unaware (right) optimal designs.}
\label{fig:uncertainty_reduction}
\end{figure}
\begin{figure}[ht!]
\centering
\includegraphics[height=.3\textwidth]{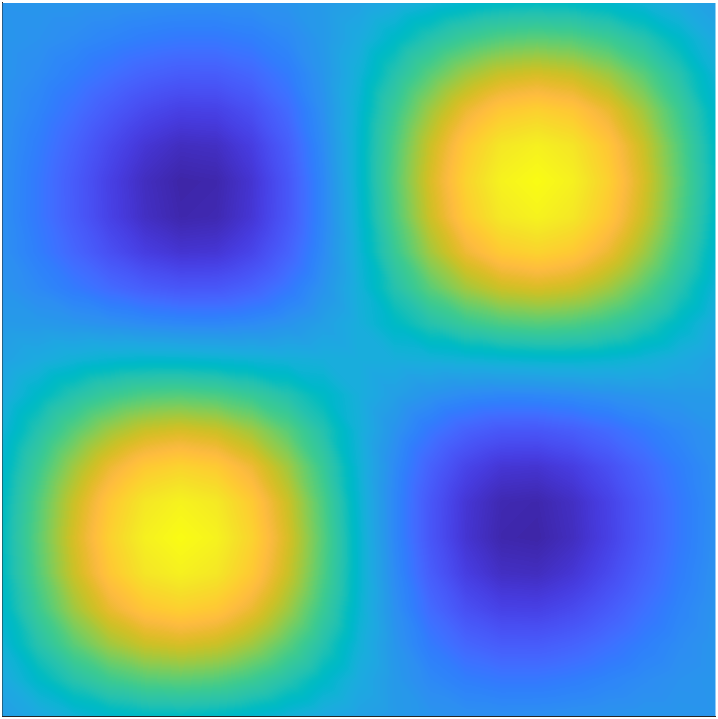}
\hfill
\includegraphics[height=.3\textwidth]{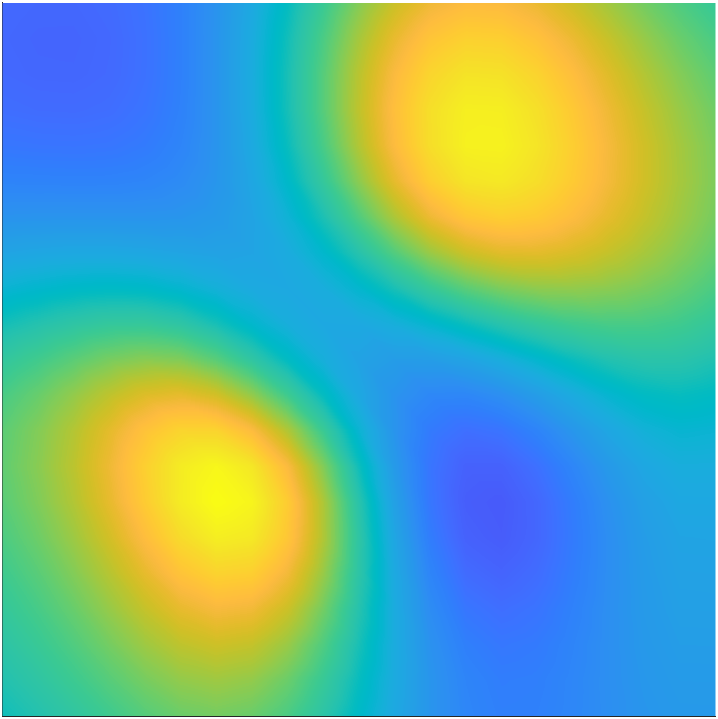}
\hfill
\includegraphics[height=.3\textwidth]{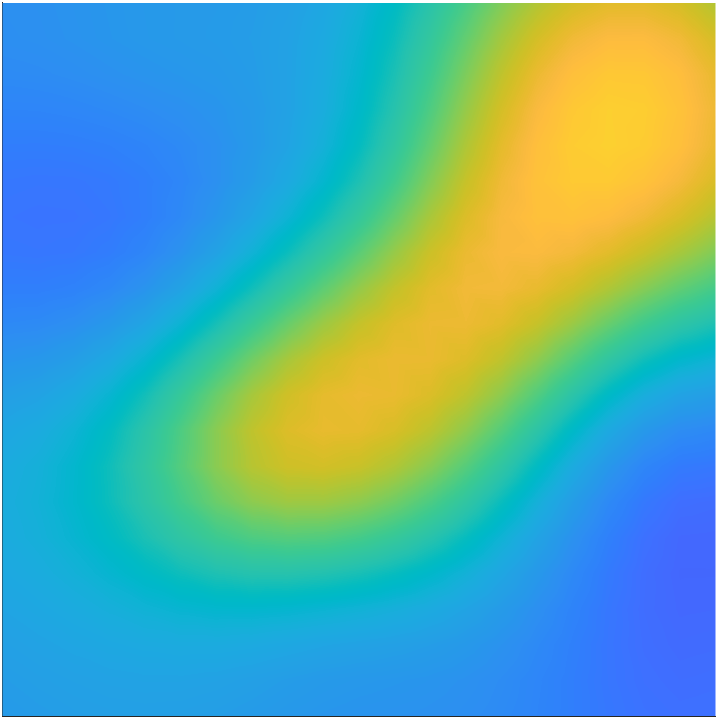}
\hfill
\includegraphics[height=.3\textwidth]{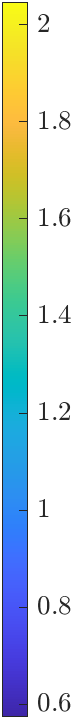}
\caption{The quality of the MAP points computed using optimal designs with 10 sensors.
We show the true parameter field (left), the MAP point computed using the uncertainty aware optimal
design (middle) and using the uncertainty unaware design (right).
}
\label{fig:MAP_points}
\end{figure}
In the above experiments, when examining the performance of uncertainty unaware
designs, model uncertainty was still accounted for (following the BAE
framework) when solving the inverse problem using these designs.  
Next, we examine the impact of ignoring model uncertainty in both OED and
inference stages.  For this experiment, we use the same synthesized data as
that used to obtain the results in \Cref{fig:MAP_points}~(right).
In~\Cref{fig:ignoring} we report the result of solving the inverse problem with
an uncertainty unaware design, when ignoring model uncertainty in the inverse
problem as well.  We note an impressive reduction of uncertainty;
see~\Cref{fig:ignoring}~(left), which uses the same scale as the standard
deviation plots in~\Cref{fig:uncertainty_reduction}. On the other hand, the MAP
point computed in this case is of very poor quality;
see~\Cref{fig:ignoring}~(right), where we have used the same scale as the MAP
point plots in~\Cref{fig:MAP_points}.  Intuitively, this indicates that
ignoring model
uncertainty in both OED and inference stages, in presence of significant
modeling uncertainties,  can lead to an unfortunate situation where one is
highly certain (i.e., low posterior variance) about a very wrong parameter
estimate.

\def \pos {0.5\columnwidth}
\addtolength\abovecaptionskip{-5pt}
\begin{figure}[ht]\centering
  \begin{tikzpicture}
    \node (1) at (0*\pos-0.4*\pos,  0*\pos){\includegraphics[height=.3\textwidth]{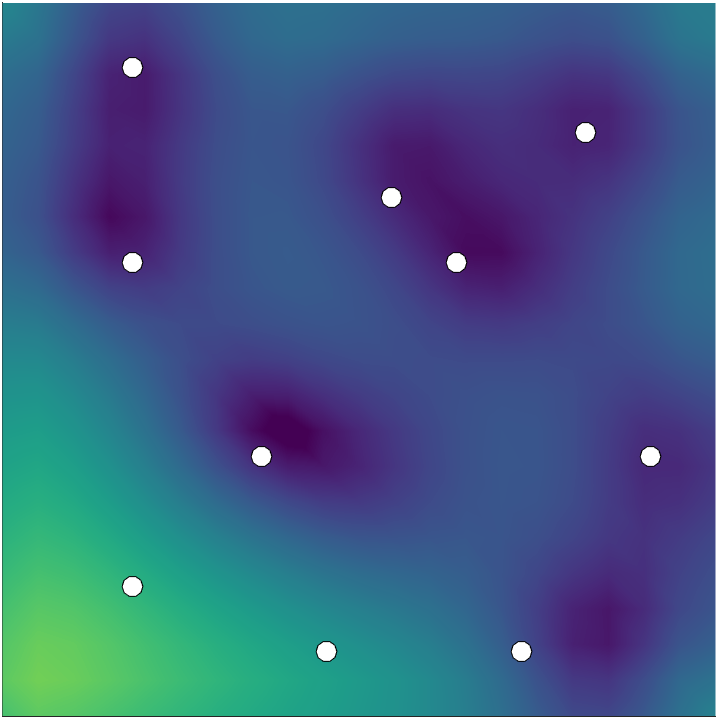}};
    \node (2) at (0*\pos-0.03*\pos, 0*\pos){\includegraphics[height=.306\textwidth]{./m_cov_CB.png}};
    \node (3) at (0*\pos+0.4*\pos,  0*\pos){\includegraphics[height=.3\textwidth]{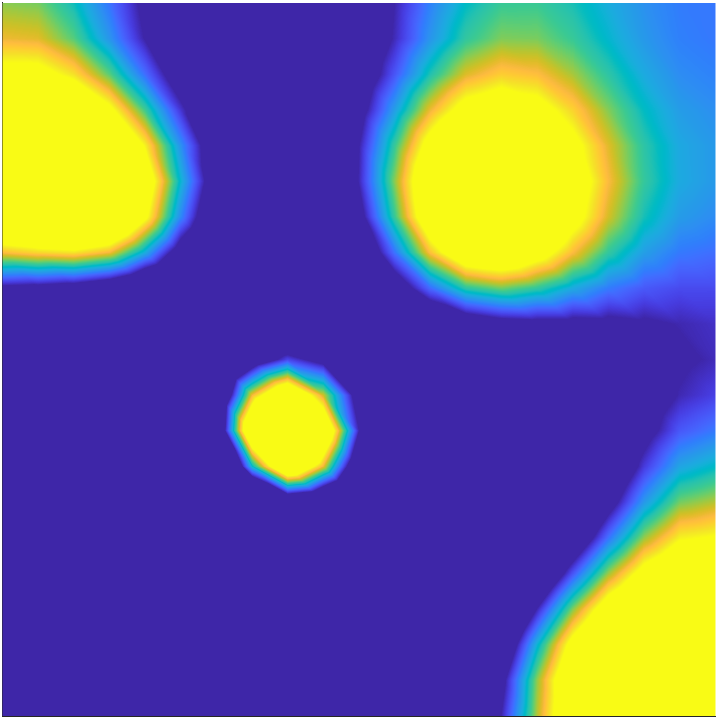}};
    \node (4) at (0*\pos+0.77*\pos, 0*\pos){\includegraphics[height=.3\textwidth]{./MAP_CB.png}};
  \end{tikzpicture}
\caption{The posterior standard deviation field (left) and MAP point (right) 
when solving the uncertainty unaware inverse problem using an uncertainty unaware 
design.}
\label{fig:ignoring}
\end{figure}
\addtolength\abovecaptionskip{5pt}

\begin{figure}
\centering
\includegraphics[height=.4\textwidth]{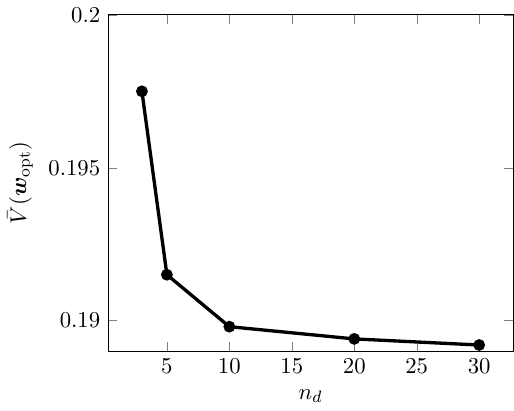}
\caption{The impact of the number of training data samples $\Nd$ on the quality of the computed
optimal design with 10 sensors.}
\label{fig:nd_study}
\end{figure}
As mentioned earlier, we used $\Nd = 5$ training data samples when solving the
OED problem \cref{equ:oeduu-ex}.  Clearly, increasing $\Nd$ would improve the
quality of the design---we would be optimizing a more accurate estimate of the
OED objective.  However, this comes with increased computational overhead. In
practice, typically a small $\Nd$ is effective in obtaining good quality
designs.  This is demonstrated in the results above (e.g.,~\cref{fig:cloud}). 
To examine the impact of changing $\Nd$ on the quality of the computed optimal
design, we perform a brief computational study
where we focus on designs with 10 sensors.
We compute the expected value of the posterior variance (as
in~\cref{equ:diagnostics}) at the computed optimal design, $\vec{w}_\text{opt}$, 
with values of $\Nd
\in \{3, 5, 10, 20, 30\}$.
The results are reported in~\cref{fig:nd_study}.  As before,
$\bar{V}(\vec{w}_\text{opt})$ is computed with $\Nv = 100$ validation data
samples.  Note that going beyond $\Nd = 10$ results in diminishing returns.
This, however, entails increased computational cost. In practice, studies such
as the one conducted here may be necessary to estimate a reasonable choice for
$\Nd$. In the present example, we see that $\Nd = 5$ is approximately at the
``elbow'' of the curve in~\cref{fig:nd_study}, making it a reasonable practical
choice. 

\section{Conclusions}\label{sec:conclusion}
In the present work, we addressed OED under uncertainty for Bayesian nonlinear
inverse problems governed by PDEs with infinite-dimensional inversion and
secondary parameters.  We have presented a mathematical framework and scalable
computational methods for computing uncertainty aware optimal designs.  Our
results demonstrate that ignoring the uncertainty in the OED
and/or the parameter inversion stages can lead to inferior designs and
inaccurate parameter estimation. Hence, it is important to account for
modeling uncertainties in Bayesian inversion and OED.

The limitations of the proposed approach are in its reliance on Gaussian
approximations for the posterior and the approximation error. The former is a
common approach in large-scale Bayesian inverse problems as well as in the BAE
literature. The Gaussian approximation to the posterior is suitable if a
linearization of the forward model, at the MAP point, is sufficiently accurate
for the set of parameters with significant posterior probability.  On the other
hand, the Gaussian approximation of the approximation error, which is guided by
the BAE approach, might fail to adequately capture the distribution of the
approximation error. However, as shown in various studies in the BAE
literature, this Gaussian approximation is reasonable in broad classes of
inverse problems; see,
e.g.,~\cite{KolehmainenTarvainenArridgeEtAl11,KaipioKolehmainen13,
MozumderTarvainenArridgeEtAl16,NicholsonPetraKaipio18}.  Additionally, in the
present work we considered a greedy approach for tackling the binary OED
optimization problems.  This can become expensive if the number of candidate
sensor locations or the desired number of sensors in the computed sensor
placements become very large.

The numerical experiments in the present work focus on an academic, albeit
application-driven, model problem. The present application was chosen since it
exhibits key problem structures seen in broad classes of ill-posed inverse
problems. Namely, smoothing properties of the forward operator lead to low-rank
structures that can be exploited when developing numerical methods. The present
numerical study also builds on and complements the previous study of BAE for the
Robin boundary condition inversion problems in~\cite{NicholsonPetraKaipio18}.
Examining the properties of the present method in more challenging inverse
problems with multiple types of modeling uncertainties is a subject for future
work.  

The discussions in this article point to a number of opportunities for future
work. In the first place, we point out that the use of BAE approach to account
for modeling uncertainties in Bayesian inverse problems is only one possible
application of this approach.  In general, BAE has been used in a wide range of
applications to account for uncertainty due to the use of approximate forward
models. For example, BAE can be used to model the approximation errors due to
the use of reduced order models or upscaled models.  BAE has also been used to
account for the errors due to the use of a mean-field model instead of an
underlying high-fidelity stochastic model~\cite{SimpsonBakerBuenzliEtAl22}.
Thus, the framework presented for OED under uncertainty in the present work can
be extended to OED in inverse problems where the forward model is replaced with
a low-fidelity approximate model instead of a computationally intensive
high-fidelity model, as long the approximation error can be modeled adequately
by a Gaussian.

Another interesting line of inquiry involves replacing the greedy strategy used
to find optimal designs with more powerful optimization algorithms. One
possibility is to follow a relaxation
approach~\cite{LiuChepuriFardad16,AlexanderianPetraStadlerEtAl14,
AlexanderianPetraStadlerEtAl16,AttiaConstantinescu22} where the design weights
are allowed to take values in the interval $[0, 1]$. This enables use of
efficient gradient-based optimization algorithms and can be combined with a
suitable penalty approach to control the sparsity of the computed sensor
placements.  An attractive alternative is the approach
in~\cite{AttiaLeyfferMunson22}, which tackles the binary OED optimization
problem by replacing it with a related stochastic programming problem.  A
further line of inquiry is investigating the idea of using fixed MAP points,
computed prior to solving the OED problem, as done in~\cite{WuChenGhattas23} for
OED problems with no additional uncertainties.  This idea, at the expense of
further approximations, replaces the bilevel optimization problem with a simpler
one, hence reducing computational cost significantly. 

Finally, in inverse problems governed by complex models with multiple sources of
secondary uncertainty, an a priori sensitivity analysis of the Bayesian inverse
problem may be necessary to identify secondary modeling uncertainties that are
most influential to the solution of the inverse problem. This may be
accomplished by suitable adaptations of hyper-differential sensitivity analysis
methods for inverse 
problems~\cite{SunseriHartVanBloemenWaandersAlexanderian20,ReeseHartBartEtAl24,SunseriAlexanderianHartEtAl24}.

\ack
The work of N.~Petra was supported in part by US National Science Foundation
grant DMS \#1723211. The work of A.~Alexanderian was supported in part by US
National Science Foundation grants DMS \#1745654 and \#2111044.

\section*{References}
\bibliographystyle{abbrv}
\bibliography{refs}

\end{document}